\documentclass[12pt,reqno]{amsart}
\usepackage{fullpage,url,amssymb,enumerate,colonequals}
\usepackage[all]{xy} 
\usepackage{mathrsfs} 

\usepackage[OT2,T1]{fontenc}
\DeclareSymbolFont{cyrletters}{OT2}{wncyr}{m}{n}
\DeclareMathSymbol{\Sha}{\mathalpha}{cyrletters}{"58}

\usepackage{color}

\newcommand{\defi}[1]{\textsf{#1}} 

\newcommand{\Schubert}{\mathcal{S}}

\newcommand{\Aff}{{\mathbb A}}
\newcommand{\Abar}{\overline{A}}
\newcommand{\Bbar}{\overline{B}}

\newcommand{\Dbar}{\overline{D}}
\newcommand{\F}{{\mathbb F}}
\newcommand{\G}{{\mathbb G}}

\newcommand{\Nbar}{\overline{N}}

\newcommand{\Q}{{\mathbb Q}}
\newcommand{\R}{{\mathbb R}}
\newcommand{\Z}{{\mathbb Z}}

\newcommand{\Vbar}{{\overline{V}}}
\newcommand{\Wbar}{{\overline{W}}}

\newcommand{\Zhat}{{\widehat{\Z}}}
\newcommand{\Zbar}{{\overline{Z}}}

\newcommand{\ksep}{{k^{\operatorname{sep}}}}

\newcommand{\Adeles}{{\mathbf A}}

\newcommand{\mm}{{\mathfrak m}}

\newcommand{\calA}{{\mathcal A}}

\newcommand{\calO}{{\mathcal O}}

\newcommand{\calZ}{{\mathcal Z}}

\newcommand{\Adist}{{\mathscr A}}
\newcommand{\EE}{{\mathscr E}}

\newcommand{\GG}{{\mathscr G}}

\newcommand{\Qdist}{{\mathscr Q}}

\newcommand{\sequence}{{\mathscr S}}
\newcommand{\Tdist}{{\mathscr T}}

\DeclareMathOperator{\Graph}{graph}
\DeclareMathOperator{\alt}{alt}

\DeclareMathOperator{\parity}{parity}
\DeclareMathOperator{\cyc}{cyc}
\DeclareMathOperator{\even}{even}
\DeclareMathOperator{\odd}{odd}
\DeclareMathOperator{\diag}{diag}
\DeclareMathOperator{\Surj}{Surj}
\DeclareMathOperator{\Prob}{Prob}
\DeclareMathOperator{\sat}{sat}

\DeclareMathOperator{\Stab}{Stab}
\DeclareMathOperator{\coker}{coker}

\DeclareMathOperator{\rk}{rk}

\DeclareMathOperator{\Char}{char}

\DeclareMathOperator{\res}{res}

\DeclareMathOperator{\im}{im}

\DeclareMathOperator{\Hom}{Hom}

\DeclareMathOperator{\Aut}{Aut}
\DeclareMathOperator{\AUT}{\bf Aut}
\DeclareMathOperator{\Gal}{Gal}

\DeclareMathOperator{\Gr}{Gr}
\DeclareMathOperator{\OGr}{OGr}

\DeclareMathOperator{\Sel}{Sel}
\DeclareMathOperator{\Seq}{Seq}

\DeclareMathOperator{\Spec}{Spec}

\DeclareMathOperator{\Frac}{Frac}

\newcommand{\tors}{{\operatorname{tors}}}

\newcommand{\tH}{{\operatorname{th}}}

\newcommand{\HH}{{\operatorname{H}}}

\newcommand{\Orthogonal}{\operatorname{O}}
\newcommand{\GL}{\operatorname{GL}}
\newcommand{\SL}{\operatorname{SL}}
\newcommand{\Sp}{\operatorname{Sp}}
\newcommand{\PGL}{\operatorname{PGL}}

\newcommand{\To}{\longrightarrow}

\newcommand{\injects}{\hookrightarrow}
\newcommand{\isom}{\simeq}

\newcommand{\Intersection}{\bigcap} 
\newcommand{\intersect}{\cap} 
\newcommand{\Union}{\bigcup} 
\newcommand{\tensor}{\otimes} 
\newcommand{\directsum}{\oplus} 
\newcommand{\Directsum}{\bigoplus} 

\newcommand{\isomto}{\overset{\sim}{\rightarrow}}

\newcommand{\stratum}{V}
\newcommand{\Zpstratum}{\calA}

\newtheorem{theorem}{Theorem}[section]
\newtheorem{lemma}[theorem]{Lemma}
\newtheorem{corollary}[theorem]{Corollary}
\newtheorem{proposition}[theorem]{Proposition}

\theoremstyle{definition}

\newtheorem{question}[theorem]{Question}
\newtheorem{conjecture}[theorem]{Conjecture}

\theoremstyle{remark}
\newtheorem{remark}[theorem]{Remark}

\usepackage[
	backref,
	pdfauthor={Bjorn Poonen}, 
]{hyperref}
\usepackage[alphabetic,backrefs,lite]{amsrefs} 

\begin{document}

\title{Modeling the distribution of ranks, Selmer groups, \\and Shafarevich--Tate groups of elliptic curves}
\author{Manjul Bhargava}
\address{Department of Mathematics, Princeton University, Princeton, NJ 08544, USA}
\email{bhargava@math.princeton.edu}

\author{Daniel M. Kane}
\address{Department of Mathematics, Stanford University, Stanford, CA 94305, USA}
\email{dankane@math.stanford.edu}
\urladdr{http://math.stanford.edu/~dankane/}

\author{Hendrik W. Lenstra jr.}
\address{Mathematisch Instituut, Universiteit Leiden, Postbus 9512, 2300 RA Leiden, The Netherlands}
\email{hwl@math.leidenuniv.nl}
\urladdr{http://www.math.leidenuniv.nl/~hwl/}

\author{Bjorn Poonen}
\address{Department of Mathematics, Massachusetts Institute of Technology, Cambridge, MA 02139-4307, USA}
\email{poonen@math.mit.edu}
\urladdr{http://math.mit.edu/~poonen/}

\author{Eric Rains}
\address{Department of Mathematics, California Institute of Technology, Pasadena, CA 91125}
\email{rains@caltech.edu}

\thanks{M.B.\ was supported by the National Science Foundation grant DMS-1001828.  D.K.\ was supported by a National Science Foundation Graduate Fellowship.  B.P.\ was supported by the Guggenheim Foundation and National Science Foundation grants DMS-0841321 and DMS-1069236.}

\date{August 5, 2013}

\begin{abstract}
Using maximal isotropic submodules in a quadratic module over $\Z_p$, 
we prove the existence of a natural discrete probability
distribution on the set of isomorphism classes of short exact sequences
of co-finite type $\Z_p$-modules,
and then conjecture that as $E$ varies over elliptic curves
over a fixed global field $k$, the distribution of
\[
   0 \to E(k) \tensor \Q_p/\Z_p \to \Sel_{p^\infty} E \to \Sha[p^\infty] \to 0
\]
is that one.
We show that this single conjecture would explain many of the 
known theorems and conjectures on 
ranks, Selmer groups, and Shafarevich--Tate groups of elliptic curves.
We also prove the existence of a discrete probability distribution
of the set of isomorphism classes of finite abelian $p$-groups equipped
with a nondegenerate alternating pairing, 
defined in terms of the cokernel of a random alternating matrix over $\Z_p$,
and we prove that the two probability distributions are compatible
with each other and with Delaunay's predicted distribution for $\Sha$.
Finally, we prove new theorems on the fppf cohomology of elliptic curves
in order to give further evidence for our conjecture.
\end{abstract}

\maketitle

\section{Introduction}\label{S:introduction}

\subsection{Selmer and Shafarevich--Tate groups}
\label{S:Selmer and Sha}

Fix a global field $k$.
Let $\Omega$ be the set of nontrivial places of $k$.
Let $\EE$ be the set of elliptic curves over $k$,
or more precisely, 
a set containing one representative of each isomorphism class.
Given $E \in \EE$ and a positive integer $n$,
the \defi{$n$-Selmer group} $\Sel_n E$ is a finite group that is used
to bound the rank of the finitely generated abelian group $E(k)$.
If $n$ is a product of prime powers $p^e$,
then $\Sel_n E$ is the direct sum of the $\Sel_{p^e}$,
so we focus on the latter groups.
If $p$ is prime, one may also form the direct limit
$\Sel_{p^\infty} E \colonequals \varinjlim \Sel_{p^e} E$.
This group, together with the $p$-primary subgroup of 
the \defi{Shafarevich--Tate group} $\Sha = \Sha(E) \colonequals 
\ker\left( \HH^1(k,E) \to \prod_{v \in \Omega} \HH^1(k_v,E) \right)$,
fits into an exact sequence
\begin{equation}
  \label{E:Seq_E}
\tag{$\textup{Seq}_E$}  
	0 \To E(k) \tensor \frac{\Q_p}{\Z_p} \To \Sel_{p^\infty} E 
	\To \Sha[p^\infty] \To 0
\end{equation}
of $\Z_p$-modules.

\begin{question}
Given a short exact sequence $\sequence$ of $\Z_p$-modules,
what is the probability that $\Seq_E \isom \sequence$
as $E$ varies over $\EE$, ordered by height?
\end{question}

Our goal is to formulate a conjectural answer
and to prove that it would imply
many of the known theorems and conjectures
on ranks, Selmer groups, and Shafarevich--Tate groups
of elliptic curves.
For example, it would imply that asymptotically 
50\% of elliptic curves over $k$ have rank~$0$,
50\% have rank~$1$, and 
0\% have rank $2$ or more: 
see Section~\ref{S:rank}.

\subsection{Intersection of random maximal isotropic \texorpdfstring{$\Z_p$}{Zp}-modules}
\label{S:intersection}

Let $n \in \Z_{\ge 0}$.
Equip $V \colonequals \Z_p^{2n}$ 
with the standard hyperbolic quadratic form $Q \colon V \to \Z_p$
given by
\begin{equation}
  \label{E:hyperbolic form}
	Q(x_1,\ldots,x_n,y_1,\ldots,y_n) \colonequals \sum_{i=1}^n x_i y_i.
\end{equation}
A $\Z_p$-submodule $Z$ of $V$ is called \defi{isotropic} if $Q|_Z=0$.
Let $\OGr_V(\Z_p)$ be the set of maximal isotropic direct summands $Z$ of $V$;
each such $Z$ is free of rank~$n$.
There is a natural probability measure on $\OGr_V(\Z_p)$, defined so that 
for each $e \ge 0$, the
distribution of $Z/p^e Z$ in $V/p^e V$ is uniform among all possibilities
(see Sections \ref{S:measures} and~\ref{S:orthogonal Grassmannian}).
Choose $Z,W \in \OGr_V(\Z_p)$ independently at random.
View $Z \tensor \frac{\Q_p}{\Z_p}$ and $W \tensor \frac{\Q_p}{\Z_p}$
as $\Z_p$-submodules of $V \tensor \frac{\Q_p}{\Z_p}$,
where the tensor products are over $\Z_p$.
Define
\[
	R \colonequals (Z \intersect W) \tensor \frac{\Q_p}{\Z_p}
	\quad \text{and} \quad
	S \colonequals \left( Z \tensor \frac{\Q_p}{\Z_p} \right)
		\intersect \left( W \tensor \frac{\Q_p}{\Z_p} \right)
\]
and define $T$ to complete an exact sequence
\[
	0 \to R \to S \to T \to 0.
\]
Then $R$ is a finite power of $\Q_p/\Z_p$,
the module $T$ is finite, and the sequence splits: 
see Section~\ref{S:RST}.
In particular, each of the $\Z_p$-modules $R$, $S$, $T$ is of co-finite type.
(A $\Z_p$-module $M$ is of \defi{co-finite type} 
if its Pontryagin dual is finitely generated over $\Z_p$,
or equivalently if $M$ is isomorphic
to $(\Q_p/\Z_p)^s \directsum F$ for some $s \in \Z_{\ge 0}$
and finite abelian $p$-group $F$.)

\begin{theorem}
\label{T:limit of RST distribution}
\hfill
\begin{enumerate}[\upshape (a)]
\item 
As $Z$ and $W$ vary, the sequence $0 \to R \to S \to T \to 0$
defines a discrete probability distribution $\Qdist_{2n}$
on the set of isomorphism classes of short exact sequences
of co-finite type $\Z_p$-modules.
\item
The distributions $\Qdist_{2n}$ converge 
to a discrete probability distribution $\Qdist$ as $n \to \infty$.
\end{enumerate}
\end{theorem}

(Here $\Qdist$ is for ``quadratic''.)
Theorem~\ref{T:limit of RST distribution}
will be proved in Section~\ref{S:comparison}.

\subsection{The model}
\label{S:model}

Let $E \in \EE$, and let $r$ be the rank of $E(k)$.
Each term in $\Seq_E$ is a co-finite type $\Z_p$-module.
In fact, 
$E(k) \tensor \frac{\Q_p}{\Z_p} \isom \left(\frac{\Q_p}{\Z_p} \right)^r$
and $\Sha[p^\infty]$ is conjecturally finite.
Moreover, since $E(k) \tensor \frac{\Q_p}{\Z_p}$ is divisible,
the sequence~\eqref{E:Seq_E} splits.

\begin{conjecture}
\label{C:main}
Fix a global field $k$.
For each short exact sequence $\sequence$ of $\Z_p$-modules,
the density of $\{E \in \EE : \Seq_E \isom \sequence \}$
equals the $\Qdist$-probability of $\sequence$.
\end{conjecture}

In other words, the sequence $0 \to R \to S \to T \to 0$ models $\Seq_E$.
In particular, $R$ conjecturally measures 
the rank of the group of rational points,
$S$ models the $p^\infty$-Selmer group,
and $T$ models the $p$-primary part of the Tate--Shafarevich group.

\begin{remark}
To define density of a subset of $\EE$ precisely, 
one orders $\EE$ by height as explained in the introductions to 
\cite{Bhargava-Shankar-preprint1} and \cite{Poonen-Rains2012-selmer}.
The reason for ordering $\EE$ by height 
is that most other orderings lead to statements that are
difficult to corroborate:
for instance, the asymptotic behavior of the number of
elliptic curves over $\Q$ of conductor up to $X$
is unknown, even when no condition on its Selmer sequence is imposed.)
\end{remark}

\begin{remark}
Certain restricted families of elliptic curves 
can exhibit very different Selmer group behavior.
The average size of $\Sel_2 E$ can even be infinite in certain families.
See \cite{Yu2005}, \cite{Xiong-Zaharescu2009}, 
\cite{Xiong-Zaharescu2008}, and \cite{Feng-Xiong2012} 
for work in this direction.
\end{remark}

We will prove that Conjecture~\ref{C:main} has the following consequences,
the first of which was mentioned already:
\begin{itemize}
\item 
Asymptotically, 50\% of elliptic curves over $k$ have rank~$0$,
and 50\% have rank~$1$; 
cf.~\cite{Goldfeld1979}*{Conjecture~B} 
and~\cites{Katz-Sarnak1999a,Katz-Sarnak1999b}.
\item 
$\Sha[p^\infty]$ is finite for 100\% of elliptic curves over $k$.
\item 
Conjecture~1.1(a) of \cite{Poonen-Rains2012-selmer} concerning
the distribution of $\Sel_p E$ holds.
In fact, our Conjecture~\ref{C:main} implies a generalization
concerning the distribution of $\Sel_{p^e} E$ for every $e \ge 0$ 
(see Section~\ref{S:p^e}).
These consequences are consistent with the partial results that
have been proved: see the introduction of \cite{Poonen-Rains2012-selmer} 
for discussion.
\item 
Delaunay's conjecture in \cites{Delaunay2001,Delaunay2007,Delaunay-Jouhet-preprint} \`a la Cohen--Lenstra 
regarding the distribution of $\Sha[p^\infty]$
for rank $r$ elliptic curves over $\Q$ 
holds for $r=0$ and $r=1$.\footnote{For this conjecture,
see \cite{Delaunay2001}*{Heuristic Assumption}, 
with the modification that $u/2$ is replaced by $u$,
as suggested by the $u=1$ case discussed in~\cite{Delaunay2007}*{\S3.2} 
(his $u$ is our $r$);
see also \cite{Poonen-Rains2012-selmer}*{Section~6} and \cite{Delaunay-Jouhet-preprint}*{Section~6.2}.
Strictly speaking, in order to have our model match Delaunay's conjecture,
we modify his conjecture to order elliptic curves over $\Q$ by height
instead of conductor.}
\end{itemize}

For $r \ge 2$,
Conjecture~\ref{C:main} cannot say anything
about the distribution of $\Sha[p^\infty]$ 
as $E$ varies over the set $\EE_r \colonequals \{E \in \EE: \rk E(k) = r\}$,
because the locus of $(Z,W) \in \OGr_V(\Z_p)^2$
where $\rk(Z \intersect W)=r$ 
is of measure $0$ (Proposition~\ref{P:rank 0 or 1}).
On the other hand, that locus carries another natural probability measure, 
so we may formulate a variant of Theorem~\ref{T:limit of RST distribution}:

\begin{theorem}
\label{T:T distribution}
\hfill
\begin{enumerate}[\upshape (a)]
\item 
If we choose $(Z,W)$ at random from the locus in $\OGr_V(\Z_p)^2$
where $\rk(Z \intersect W)=r$,
then the isomorphism type of $T$ 
is given by a discrete probability distribution $\Tdist_{2n,r}$.
\item
The distributions $\Tdist_{2n,r}$ converge to a limit $\Tdist_r$
as $n \to \infty$.
\item 
The distribution $\Tdist_r$ is the same as the distribution
in Delaunay's conjecture on $\Sha[p^\infty]$ for rank $r$ elliptic curves
over $\Q$.
\end{enumerate}
\end{theorem}

Theorem~\ref{T:T distribution} will be proved in Section~\ref{S:comparison}.

\begin{conjecture}
\label{C:Sha for rank r}
Fix a global field $k$ and $r \in \Z_{\ge 0}$ such that $\EE_r$ is infinite.
For each finite abelian $p$-group $G$,
the density of $\{E \in \EE_r : \Sha[p^\infty] \isom G \}$ in $\EE_r$
equals the $\Tdist_r$-probability of $G$.
\end{conjecture}

\begin{remark}
In fact, Delaunay made predictions for
the whole group $\Sha$ and not only $\Sha[p^\infty]$ for one $p$ at a time.
We will formulate a corresponding model,
and prove its compatibility with Delaunay's: see Section~\ref{S:Zhat}.
\end{remark}

\begin{remark}
In fact, one can turn things around, 
and use Delaunay's conjecture for $\Sha[p^\infty]$ together with the 
conjecture that the rank $r$ is $0$ or $1$ with 50\% probability each
to obtain conjectural distributions for $\Sel_{p^\infty} E$ and $\Sel_{p^e} E$ 
(cf. \cite{Delaunay-Jouhet-preprint}*{\S 6.2}
and \cite{Delaunay-Jouhet-preprint2}*{\S 5}).
Specifically, 
$\Sel_{p^\infty} E \isom (\Q_p/\Z_p)^r \directsum \Sha[p^\infty]$;
if moreover $E(k)_{\tors}=0$, 
as holds for 100\% of elliptic curves (Lemma~\ref{L:torsion=0}),
then $\Sel_{p^\infty} E$ determines $\Sel_{p^e} E$ 
(Proposition~\ref{P:small Selmer, big Selmer}\eqref{I:Sel_q in Sel_p^infty}).
\end{remark}

\subsection{Cokernel of a random alternating matrix}
\label{S:cokernel}

Moreover, Conjecture~\ref{C:Sha for rank r} implies that
another natural distribution on finite abelian $p$-groups 
yields a model for $\Sha[p^\infty]$.
We now describe it.
For an even integer $n$, 
choose an alternating $n\times n$ matrix $A \in M_n(\Z_p)$ 
(see Section~\ref{S:pairings} for definitions)
at random with respect to Haar measure,
and let $\Adist_{n,0}$ be the distribution of $\coker A$.

More generally, for any fixed $r \ge 0$, 
for $n$ with $n-r \in 2\Z_{\ge 0}$,
choose $A$ at random from the set of alternating matrices in $M_n(\Z_p)$
such that $\rk A = n-r$ (with respect to a measure to be described), 
and let $\Adist_{n,r}$ be the distribution of $(\coker A)_{\tors}$.

\begin{theorem}
\label{T:A=T}
For each $r \ge 0$, 
\begin{enumerate}[\upshape (a)]
\item \label{I:Adist exists}
the distributions $\Adist_{n,r}$ converge to a limit $\Adist_r$
as $n \to \infty$ through integers with $n-r \in 2\Z_{\ge 0}$, and
\item \label{I:Adist=Tdist}
the distributions $\Adist_r$ and $\Tdist_r$ coincide.
\end{enumerate}
\end{theorem}

\begin{remark}
Delaunay's conjecture for $\Sha$ 
was made in analogy with 
the Cohen--Lenstra heuristics for class groups~\cite{Cohen-Lenstra1983}.
Later, Friedman and Washington~\cite{Friedman-Washington1989} 
recognized the Cohen--Lenstra distribution on $p$-primary parts
as being the distribution of the cokernel of a random matrix over $\Z_p$.
This was our motivation for modeling $\Sha[p^\infty]$ as the cokernel of
a random \emph{alternating} matrix over $\Z_p$.
\end{remark}

\subsection{Cassels--Tate pairing}
\label{S:Cassels-Tate}

The Cassels--Tate pairing on $\Sha$ is an alternating bilinear pairing
\[
	\Sha \times \Sha \to \Q/\Z
\]
whose kernel on each side is the maximal divisible subgroup of $\Sha$.
Taking $p$-primary parts yields an alternating bilinear pairing
\[
	\Sha[p^\infty] \times \Sha[p^\infty] \to \Q_p/\Z_p.
\]
If $\Sha[p^\infty]$ is finite, then this pairing is nondegenerate.

Recall that we are modeling $\Sha[p^\infty]$ by 
the random groups $T$ and $(\coker A)_{\tors}$ 
constructed in Sections~\ref{S:intersection} and~\ref{S:cokernel},
respectively.
As evidence that these models are reasonable,
we will construct a canonical nondegenerate alternating pairing
on each of $T$ and $(\coker A)_{\tors}$: 
see Section~\ref{S:model for Cassels-Tate}.

\subsection{Arithmetic justification}

We have used theorems on the arithmetic of elliptic curves
to guide the development of our models for ranks, Selmer groups, and
Shafarevich--Tate groups.
Conversely, part of the reason for developing such models 
is to suggest new structure in the arithmetic
of elliptic curves that might be discovered.

In \cite{Poonen-Rains2012-selmer},
modeling $\Sel_p E$ by an intersection of two random maximal isotropic 
$\F_p$-subspaces
was suggested by a theorem that $\Sel_p E$ \emph{is} an intersection of
two maximal isotropic subspaces 
in an infinite-dimensional quadratic space over~$\F_p$.
In Section~\ref{S:arithmetic justification},
we prove an analogue for prime powers: for $100\%$ of $E \in \EE$,
the group $\Sel_{p^e} E$ 
is isomorphic to an intersection of two maximal isotropic subgroups
of an infinite quadratic $\Z/p^e\Z$-module $\HH^1(\Adeles,E[p^e])$.
But why in our model do we assume that the two subgroups
we intersect are \emph{direct summands}?
Answer: if maximal isotropic subgroups of $(\Z/p^e\Z)^{2n}$ are to be sampled
from a distribution that is invariant under the orthogonal group,
the only distribution that yields predictions compatible
with the distribution in \cite{Poonen-Rains2012-selmer} modeling $\Sel_p$ 
is the distribution 
supported on direct summands (Remark~\ref{R:why direct summands}).
Moreover, we provide some justification from the arithmetic of
elliptic curves: 
we prove that one of the two subgroups of $\HH^1(\Adeles,E[p^e])$ 
actually is a direct summand (Corollary~\ref{C:adelic direct summand}),
and conjecture that the other one is too for $100\%$ of $E \in \EE$ 
(Conjecture~\ref{C:global H^1 is direct summand}).
All this is the motivation for our model for $\Sel_{p^\infty}$.

Venkatesh and Ellenberg~\cite{Venkatesh-Ellenberg2010}*{Section~4.1}
observed that the Friedman--Washington 
reinterpretation of the Cohen--Lenstra distribution
could be justified by a parallel construction in the arithmetic
of number fields, namely that the class group is the cokernel
of the homomorphism from the $S$-units to the group of fractional
ideals supported on $S$, for a suitably large set of places $S$.

\begin{question}
Is there a construction in the arithmetic of elliptic curves
that realizes $\Sha$ as the torsion subgroup of the cokernel 
of a natural alternating map of free $\Z$-modules?
\end{question}

\section{The canonical measure on the set of \texorpdfstring{$\Z_p$}{Zp}-points of a scheme}
\label{S:measures}

The following is a consequence of work of Oesterl\'e and Serre.

\begin{proposition}
\label{P:measure}
Let $X$ be a finite-type $\Z_p$-scheme.
Let $d=\dim X_{\Q_p}$.
Equip $X(\Z_p)$ with the $p$-adic topology.
\begin{enumerate}[\upshape (a)]
\item\label{I:measure exists}
There exists a unique bounded $\R_{\ge 0}$-valued measure $\mu = \mu_X$ on
the Borel $\sigma$-algebra of $X(\Z_p)$
such that for any open and closed subset $S$ of $X(\Z_p)$, we have
\[
	\mu(S) = \lim_{e \to \infty}
	\frac{\#(\textup{image of $S$ in $X(\Z/p^e \Z)$})}{(p^e)^d}.
\]
\item\label{I:lower dimensional}
If $Y$ is a subscheme of $X$ and $\dim Y_{\Q_p} < d$,
then $\mu(Y(\Z_p))=0$.
\item\label{I:positive measure}
If $S$ is an open subset of $X(\Z_p)$,
and $X_{\Q_p}$ is smooth of dimension $d$ at $s_{\Q_p}$ for some $s \in S$,
then $\mu(S)>0$.
\end{enumerate}
\end{proposition}

\begin{proof}\hfill
\begin{enumerate}[\upshape (a)]
\item
If $X$ is affine, this is a consequence of 
the discussion surrounding Th\'eor\`eme~2 of~\cite{Oesterle1982},
which builds on~\cite{Serre1981}*{S3}, 
and the Hahn--Kolmogorov extension theorem.
In general, let $(X_i)$ be a finite affine open cover of $X$.
Each set $X_i(\Z_p)$ is open \emph{and closed} in $X(\Z_p)$,
because $X_i(\Z_p)$ equals the inverse image of $X_i(\F_p)$ under the reduction
map $X(\Z_p) \to X(\F_p)$.
Since $\Z_p$ is a local ring, the sets $X_i(\Z_p)$ form a cover of $X(\Z_p)$.
The measures on $X_i(\Z_p)$ and $X_j(\Z_p)$ are compatible on the
intersection, by uniqueness, so they glue to give the required measure
on $X(\Z_p)$.
\item
We may assume that $X$ is affine and that $Y$ is a closed subscheme of $X$.
Even though $Y(\Z_p)$ might not be open in $X(\Z_p)$,
it is an analytic closed subset (see~\cite{Oesterle1982}*{\S2}),
so 
\[
	\mu(Y(\Z_p)) = \lim_{e \to \infty}
	\frac{\# Y(\Z/p^e \Z)}{(p^e)^d}
\]
still holds.
According to \cite{Serre1981}*{p.~145, Th\'eor\`eme~8},
$\#Y(\Z/p^e \Z) = O((p^e)^{d-1})$ as $e \to \infty$,
so the limit is $0$.
\item
See the discussion before Th\'eor\`eme~2 of~\cite{Oesterle1982}.\qedhere
\end{enumerate}
\end{proof}

\begin{corollary}
\label{C:prob measure}
If $X$ is as in Proposition~\ref{P:measure},
and $X_{\Q_p}$ is smooth of dimension $d$ at $x_{\Q_p}$ for some $x \in X(\Z_p)$,
then $\mu$ can be normalized to yield
a probability measure $\nu$ on $X(\Z_p)$.
\end{corollary}

{}From now on, when we speak of choosing an element of $X(\Z_p)$
uniformly at random for $X$ as in Corollary~\ref{C:prob measure},
we mean choosing it according to the measure $\nu$.

\section{Modeling Shafarevich--Tate groups using alternating matrices}

\subsection{Notation}
\label{S:notation}

Let $R$ be a principal ideal domain.
(We could work with more general rings, but we have no need to.)
Let $K=\Frac R$.
Given an $R$-module $L$, let $L^T\colonequals \Hom_R(L,R)$,
let $L_K\colonequals L \tensor_R K$,
and let $L_\tors\colonequals \{x \in L: rx=0 \textup{ for some nonzero $r \in R$}\}
= \ker(L \to L_K)$.
If a free $R$-module $L$ has been fixed,
and $N$ is a submodule of $L$,
define the \defi{saturation}
\[
	N^{\sat}\colonequals  N_K \intersect L = \{x \in L : rx \in N \textup{ for some nonzero $r \in R$}\}.
\]
Given a homomorphism $A \colon L \to M$,
let $A^t \colon M^T \to L^T$ denote the dual homomorphism;
this notation is compatible with
the notation $A^t$ for the transpose of a matrix.
Let $M_n(R)_{\alt}$ be the set of \defi{alternating} $n \times n$ matrices,
i.e., matrices $A$ with zeros on the diagonal satisfying $A^t=-A$.
For $S \subseteq M_n(R)$, define $S_{\alt} = S \intersect M_n(R)_{\alt}$.

\subsection{Symplectic abelian groups}

Define a \defi{symplectic abelian group} 
(called \defi{group of type $S$} in \cite{Delaunay2001})
to be a finite abelian group $G$ equipped with
a nondegenerate alternating pairing $[\;,\;]\colon G \times G \to \Q/\Z$.
An isomorphism between two symplectic abelian groups
is a group isomorphism that respects the pairings.
Let $\Sp(G)$ be the group of automorphisms of $G$ respecting $[\;,\;]$.
One can show that two symplectic abelian groups are isomorphic 
if and only if their underlying abelian groups are isomorphic.
If $p$ is a prime, define a \defi{symplectic $p$-group}
to be a symplectic abelian group whose order is a power of $p$;
in this case, $[\;,\;]$ may be viewed as taking values in $\Q_p/\Z_p$.

\subsection{Pairings on the cokernel of  an alternating matrix}
\label{S:pairings}

Let $L$ be a free $R$-module of rank $n$.
Let $\calA \colon L \times L \to R$ be an alternating $R$-bilinear pairing;
\defi{alternating} means that $\calA(x,x)=0$ for all $x \in L$.
Let $A \colon L \to L^T$ be the induced $R$-homomorphism.
If we choose a basis for $L$ and use the dual basis for $L^T$,
then $A$ corresponds to a matrix $A \in M_n(R)_{\alt}$;
then $\calA$ is identified with
\begin{align*}
	R^n \times R^n &\to R \\
	x,y &\mapsto x^t A y.
\end{align*}
A change of basis of $L$ is given by a matrix $M \in \GL_n(R)$,
which changes $A$ to $M^t A M$;
this defines an action of $\GL_n(R)$ on $M_n(R)_{\alt}$.

Let $L^\perp \colonequals \{x \in L_K \colon \calA(L,x) \subseteq R\}$.
Then $\calA_K\colonequals \calA \tensor K$ induces an alternating pairing
\begin{equation}
\label{E:L^perp/L pairing}
	\frac{L^\perp}{L} \times \frac{L^\perp}{L} \to \frac{K}{R}.
\end{equation}

\subsection{The pairing in the nonsingular case}
\label{S:nonsingular}

Suppose that $A$ is nonsingular
in the sense that $A_K \colon L_K \to (L^\vee)_K$ is an isomorphism
(i.e., $\det A \ne 0$).
Then $L^\perp = A^{-1} L^T$ and multiplication-by-$A$ induces an
isomorphism $L^\perp/L \isom L^T/AL = \coker A$ of finite torsion $R$-modules.
Substituting (and flipping a sign)
rewrites~\eqref{E:L^perp/L pairing} as an alternating pairing
\begin{equation}
\label{E:coker pairing}
	\langle \;,\;\rangle_A \colon \coker A \times \coker A \to \frac{K}{R}
\end{equation}
induced by
\begin{align}
\label{E:bracket A}
	[\;,\;]_A \colon R^n \times R^n &\to \frac{K}{R} \\
\nonumber
	x,y &\mapsto x^t A^{-1} y,
\end{align}
where each $R^n$ is $L^T$.
The right kernel of $[\;,\;]_A$ is the image $\im(A \colon R^n \to R^n)$
since one has $x^t(A^{-1} y) \in R$ for all $x \in R^n$
if and only if $A^{-1} y \in R^n$.
Since the pairing is alternating, the left kernel is the same.
Thus the left and right kernels of $\langle\;,\;\rangle_A$ are $0$;
i.e., $\langle\;,\;\rangle_A$ is nondegenerate.

\subsection{The pairing in the singular case}

Suppose that $\det A=0$.
Let $L_0= \ker A$.
The quotient $L/L_0$ is torsion-free, and hence free,
since $R$ is a principal ideal domain.
Then $\calA$ induces a nonsingular alternating pairing $\calA_1$ on $L/L_0$,
corresponding to some $A_1$.
The submodule $(L/L_0)^T$ of $L^T$ is the saturation of $\im A$ in $L^T$,
i.e., $(\im A)^{\sat} = (\im A)_K \intersect L^T$.
Then $\coker(A_1 \colon L/L_0 \to (L/L_0)^T)$
identifies with $(\im A)^{\sat}/(\im A) \isom (\coker A)_\tors$.
Applying Section~\ref{S:nonsingular} to $A_1$,
we obtain an alternating $R$-bilinear pairing
\[
	\langle \;,\; \rangle_A \colon
	(\coker A)_\tors \times (\coker A)_\tors \to \frac{K}{R}
\]
whose left and right kernels are $0$.

\subsection{Lemmas}
\label{S:lemmas}

Take $R=\Z_p$.
Given $A \in M_n(\Z_p)$, let $\Abar\colonequals A \bmod p \in M_n(\F_p)$.
If $A \in M_n(\Z_p)_{\alt}$ and $\det A \ne 0$,
then by Section~\ref{S:nonsingular},
$\coker A$ with $\langle \;,\; \rangle_A$ is a symplectic $p$-group.

\begin{lemma}
\label{L:same pairing}
Suppose that $A,D \in M_n(\Z_p)_{\alt}$, and $\det D \ne 0$.
We have $[\;,\;]_A = [\;,\;]_D$ if and only if $\det A\ne 0$
and $A^{-1}-D^{-1} \in M_n(\Z_p)$.
\end{lemma}

\begin{proof}
This is immediate from~\eqref{E:bracket A}.
\end{proof}

\begin{lemma}
\label{L:A=D+DND}
Suppose that $D \in M_n(\Z_p)$ and $\det D \ne 0$.
Then
\[
	\{ A \in M_n(\Z_p) : \det A \ne 0 \textup{ and }
		A^{-1} - D^{-1} \in M_n(\Z_p) \}
	=
	\{A \in D + D M_n(\Z_p) D : \rk \Abar = \rk \Dbar\}.
\]
\end{lemma}

\begin{proof}
Suppose that $A \in M_n(\Z_p)$ is such that $\det A \ne 0$
and $A^{-1} - D^{-1} = N$ for some $N \in M_n(\Z_p)$.
Multiplying by $A$ on the left yields $I-AD^{-1} = AN$,
so $AD^{-1} \in M_n(\Z_p)$;
similarly $DA^{-1} \in M_n(\Z_p)$, so $AD^{-1} \in \GL_n(\Z_p)$,
and in particular $\rk \Dbar = \rk \Abar$.
Multiplying instead by $D$ on the left and $A$ on the right
yields $D-A=DNA=D(N \cdot AD^{-1})D \in D M_n(\Z_p) D$,
so $A \in D + D M_n(\Z_p) D$.

Conversely, suppose that $A=D+DND$ with $N \in M_n(\Z_p)$,
and $\rk \Abar = \rk \Dbar$.
Then $\Abar = \Dbar + \Dbar \Nbar \Dbar$,
so $\ker \Dbar \subseteq \ker \Abar$,
and the rank condition implies $\ker \Dbar = \ker \Abar$.
If $v \in \ker(\overline{I+ND})$,
then $v \in \ker \Abar = \ker \Dbar$; so both $\overline{I+ND}$
and $\overline{ND}$ kill $v$, so $v=0$.
Thus $\overline{I+ND} \in \GL_n(\F_p)$,
so $I+ND \in \GL_n(\Z_p)$.
Now $D^{-1} A = I + ND$, so its inverse $A^{-1} D$ is in $\GL_n(\Z_p)$ too.
Multiplying $A=D+DND$ by $A^{-1}$ on the left and $D^{-1}$ on the right
yields $D^{-1} = A^{-1} + A^{-1} D N$,
so $A^{-1}-D^{-1} = -(A^{-1} D) N \in M_n(\Z_p)$.
\end{proof}

\begin{corollary}
\label{C:Prob A^{-1} = D^{-1}}
Let $n$ be even, let $e_1,\ldots,e_{n/2} \in \Z_{\ge 0}$,
and let
\[
D=
\begin{pmatrix}
  0 & \diag(p^{e_1},\ldots,p^{e_{n/2}}) \\
 -\diag(p^{e_1},\ldots,p^{e_{n/2}}) & 0 \\
\end{pmatrix} \in M_n(\Z_p)_{\alt}.
\]
Let $m =2\#\{i \in \{1,\ldots,n/2\} : e_i=0\}$.
If $A \in M_n(\Z_p)_{\alt}$ is chosen at random
with respect to Haar measure, then
\[
	\Prob\left(\det A \ne 0 \textup{ and }
		A^{-1} - D^{-1} \in M_n(\Z_p)\right)
	= \frac{\#\GL_m(\F_p)_{\alt}}{\#M_m(\F_p)_{\alt}} |\det D|_p^{n-1}.
\]
\end{corollary}

\begin{proof}
Let $e_{n/2+i}=e_i$ for $i=1,\ldots,n/2$.
For $A \in M_n(\Z_p)_{\alt}$,
the condition $A \in D + D M_n(\Z_p) D$
is equivalent to $a_{ij} \equiv d_{ij} \pmod{p^{e_i} p^{e_j} \Z_p}$ for all $i<j$,
so
\[
	\Prob\left(A \in D + D M_n(\Z_p) D \right)
	=\prod_{i<j} p^{-e_i} p^{-e_j}
	= \prod_{i=1}^n (p^{-e_i})^{n-1}
	= |\det D|_p^{n-1}.
\]
Let $B \in M_m(\Z_p)_{\alt}$
be the minor formed by the entries $a_{ij}$ such that $e_i=e_j=0$.
The condition $\rk \Abar = \rk \Dbar$ is equivalent to $\Bbar \in \GL_m(\F_p)$,
which is independent of the congruences above
and which holds with probability $\#\GL_m(\F_p)_{\alt}/\#M_m(\F_p)_{\alt}$.
Multiplying yields the result, by Lemma~\ref{L:A=D+DND}.
\end{proof}

Combining Corollary~\ref{C:Prob A^{-1} = D^{-1}}
with Lemma~\ref{L:same pairing} yields

\begin{corollary}
\label{C:A=D}
Retain the notation of Corollary~\ref{C:Prob A^{-1} = D^{-1}}.  Then
\[
	\Prob\left([\;,\;]_A = [\;,\;]_D\right)
	= \frac{\#\GL_m(\F_p)_{\alt}}{\#M_m(\F_p)_{\alt}} |\det D|_p^{n-1}.
\]
\end{corollary}

\begin{corollary}
\label{C:probability of pairing}
Fix {\em any} alternating pairing
$[\;,\;] \colon \Z_p^n \times \Z_p^n \to \Q/\Z$
inducing a nondegenerate pairing on a finite quotient $G$ of $\Z_p^n$.
Let $m=n-\dim_{\F_p} G[p]$.
If $A \in M_n(\Z_p)_{\alt}$ is chosen at random, then
\begin{equation}
\label{E:prob of pairing}
	\Prob\left([\;,\;]_A = [\;,\;]\right)
	= \frac{\#\GL_m(\F_p)_{\alt}}{\#M_m(\F_p)_{\alt}} (\#G)^{1-n}.
\end{equation}
\end{corollary}

\begin{proof}
The structure theorem for symplectic modules over principal ideal domains
implies that there exists a change-of-basis matrix $M \in \GL_n(\Z_p)$ and
a matrix $D$ as in Corollary~\ref{C:Prob A^{-1} = D^{-1}}
such that $[\;,\;]=[\;,\;]_{M^t D M}$.
The change of basis reduces the statement to Corollary~\ref{C:A=D}.
\end{proof}

The fraction on the right side of~\eqref{E:prob of pairing}
can be evaluated:

\begin{lemma}
\label{L:invertible alternating matrices}
For $m \in 2\Z_{\ge 0}$,
\[
	\frac{\#\GL_m(\F_p)_{\alt}}{\#M_m(\F_p)_{\alt}}
	= \prod_{i=1}^{m/2} (1-p^{1-2i}).
\]
\end{lemma}

\begin{proof}
There are $p^{m-1}-1$ possibilities for the first column of
a matrix in $\GL_m(\F_p)_{\alt}$,
and the number of choices for the rest of the matrix is independent of
this first choice,
as one sees by performing a change of basis of $\F_p^n$ 
fixing $(1,0,\ldots,0)^t$.
If the first column is $(0,1,0,0,\ldots,0)^t$,
there are $p^{m-2}$ possibilities for the second column
(it has the shape $(1,0,*,\cdots,*)$),
and then the lower $(m-2) \times (m-2)$ block is an arbitrary element
of $\GL_{m-2}(\F_p)_{\alt}$.
Thus
\[
	\#\GL_m(\F_p)_{\alt} = (p^{m-1}-1) p^{m-2} \#\GL_{m-2}(\F_p)_{\alt}.
\]
Induction on $m$ yields $\#\GL_m(\F_p)_{\alt}$,
and we divide by $\#M_m(\F_p)_{\alt} = p^{\binom{m}{2}}$.
\end{proof}

The following will be used in Section~\ref{S:comparison}.

\begin{lemma}
\label{L:kernel at least n/2}
The probability that a random $A \in M_n(\F_p)_{\alt}$
satisfies $\dim \ker A \ge n/2$
tends to $0$ as $n \to \infty$.
\end{lemma}

\begin{proof}
Let $k=\lceil n/2 \rceil$.
The number of $A \in M_n(\F_p)_{\alt}$ with $\dim \ker A \ge n/2$
is at most the number of pairs $(A,K)$ with $A \in M_n(\F_p)_{\alt}$
and $K$ a $k$-dimensional subspace of $\ker A$.
The number of $K$'s is $O(p^{k(n-k)})$.
For each $K$, the $A$'s vanishing on $K$
correspond to alternating maps from the $(n-k)$-dimensional space
$\F_p^n/K$ to its dual,
of which there are $p^{(n-k)(n-k-1)/2}$.
Thus the total number of $A$ with $\dim \ker A \ge n/2$ is at most
\[
	O(p^{k(n-k)}) \cdot p^{(n-k)(n-k-1)/2} = O(p^{(n-k)(n+k-1)/2}).
\]
Dividing by $\# M_n(\F_p)_{\alt} = p^{n(n-1)/2}$ yields $O(p^{-k(k-1)/2})$,
which tends to $0$ as $n \to \infty$.
\end{proof}

\begin{remark}
In fact, Lemma~\ref{L:kernel at least n/2} remains true
if $n/2$ is replaced by any function of $n$ tending to $\infty$,
but the $n/2$ version suffices for our application.
\end{remark}

\subsection{The distribution in the nonsingular case}

The following theorem states that the limit $\Adist_0$ 
in Theorem~\ref{T:A=T}\eqref{I:Adist exists} 
exists, and provides an explicit formula for its value:

\begin{theorem}
\label{T:u=0}
For each symplectic $p$-group $G$,
if $A$ is chosen at random in $M_n(\Z_p)_{\alt}$
with respect to Haar measure for even $n$,
then
\[
	\lim_{\substack{n \to \infty \\ n \textup{ even}}} \Prob\left( \coker A \isom G \right)
	= \frac{\# G}{\#\Sp(G)} \prod_{i=1}^\infty (1-p^{1-2i}).
\]
Moreover, the sum of the right side over all such $G$ equals $1$.
\end{theorem}

\begin{proof}
Define 
\[
	\pi_n(G) \colonequals \Prob\left( \coker A \isom G \right), \quad\textup{and}\quad \pi(G) \colonequals \lim_{\substack{n \to \infty \\ n \textup{ even}}} \pi_n(G).
\]
Let $m=n-\dim_{\F_p} G[p] \in 2\Z_{\ge 0}$.
Given a surjection $f \colon \Z_p^n \to G$,
we may pull back the pairing on $G$ to obtain an alternating pairing
on $\Z_p^n$.
This defines a surjection from $\Surj(\Z_p^n,G)$
to the set of alternating pairings
$[\;,\;] \colon \Z_p^n \times \Z_p^n \to \Q/\Z$
such that the induced $S$-group is isomorphic to $G$.
Each fiber is the orbit of a free action of $\Sp(G)$ on $\Surj(\Z_p^n,G)$
(by post-composition),
so the number of such $[\;,\;]$'s is $\# \Surj(\Z_p^n,G) / \#\Sp(G)$.
By Corollary~\ref{C:probability of pairing},
the probability that $[\;,\;]_A$ equals any fixed one of these $[\;,\;]$
equals
\[
	\frac{\#\GL_m(\F_p)_{\alt}}{\#M_m(\F_p)_{\alt}} (\#G)^{1-n}.
\]
Multiplying yields
\begin{equation}
\label{E:coker prob for finite n}
	\pi_n(G) = \frac{\# \Surj(\Z_p^n,G)}{\#\Sp(G)}
	\frac{\#\GL_m(\F_p)_{\alt}}{\#M_m(\F_p)_{\alt}} (\#G)^{1-n}.
\end{equation}
As $n \to \infty$, we have $m \to \infty$ through even integers,
and $\#\Surj(\Z_p^n,G)/(\#G)^n \to 1$
since almost all homomorphisms $\Z_p^n \to G$ are surjective,
so by Lemma~\ref{L:invertible alternating matrices},
we obtain
\[
	\pi(G) = \frac{\# G}{\#\Sp(G)} \prod_{i=1}^\infty (1-p^{1-2i}).
\]

It remains to prove that $\sum_G \pi(G)=1$.
For fixed $n$, the event that $\coker A$ is infinite
corresponds to the $\Z_p$-points of a hypersurface in 
the affine space of alternating matrices,
so Proposition~\ref{P:measure}\eqref{I:lower dimensional}
shows that it has probability~$0$; thus $\sum_G \pi_n(G)=1$.
By Fatou's lemma, $\sum_G \pi(G) \le 1$.
In particular,
\[
	\sum_G \frac{\#G}{\#\Sp(G)} \le \prod_{i=1}^\infty (1-p^{1-2i})^{-1} 
	< \infty.
\]
In \eqref{E:coker prob for finite n}, we have
$\# \Surj(\Z_p^n,G) \le \#G^n$
and $\#\GL_m(\F_p)_{\alt} \le \#M_m(\F_p)_{\alt}$,
so $\pi_n(G) \le \#G/\#\Sp(G)$;
this lets us apply the dominated convergence theorem to deduce
\[
	\sum_G \pi(G) = \sum_G \lim \pi_n(G) = \lim \sum_G \pi_n(G) = 1.\qedhere
\]
\end{proof}

\subsection{The distribution in the singular case}
\label{S:singular distribution}

Fix $n \ge 0$.
The variety $\Aff^{n(n-1)/2}$ parametrizing
alternating $n \times n$ matrices over any field
can be stratified according to the rank of the matrix.
Namely, given an integer $r$ with $0 \le r \le n$ and $n-r$ even,
let $\stratum_{n,r}$ be the locally closed subvariety
parametrizing alternating $n \times n$ matrices of rank $n-r$.
(The hypotheses on $r$ are needed
to ensure that $\stratum_{n,r}$ is nonempty.)
The existence of symplectic bases shows that $\stratum_{n,r}$
is a homogeneous space for the action of $\GL_n$ on $\Aff^{n(n-1)/2}$.

\begin{lemma}
\label{L:dim stratum}
$\dim \stratum_{n,r} = \binom{n}{2} - \binom{r}{2} \equalscolon d$.
\end{lemma}

\begin{proof}
Sending a matrix $A \in M_n(k)$ to $\im(A)$
defines a morphism from $\stratum_{n,r}$ to the Grassmannian
of $(n-r)$-planes $\Pi$ in $n$-space.
The fiber above $\Pi$ parametrizes nondegenerate alternating maps
from $k^n/(\im A)^\perp \isom \im(A)^T$ to $\im(A)$,
so each fiber has dimension $\binom{n-r}{2}$.
Thus $\dim \stratum_{n,r} = r(n-r) + \binom{n-r}{2} = \binom{n}{2} - \binom{r}{2}$.

Alternatively, one could compute the dimension of the stabilizer of
\[
\begin{pmatrix}
 & I_{(n-r)/2} & \\
-I_{(n-r)/2} & & \\
& & 0_r
\end{pmatrix} \in \stratum_{n,r}
\]
by writing an equation in $2 \times 2$ block matrices 
with blocks of size $n-r$ and $r$.
\end{proof}

We have the locally closed stratification
\[
	\Aff^{n(n-1)/2} = \Union_r \stratum_{n,r},
\]
where $r$ ranges over integers with $0 \le r \le n$ and $n-r$ even.
The Zariski closure $\overline{\stratum}_{n,r}$ of $\stratum_{n,r}$
in $\Aff^{n(n-1)/2}$ is the locus $\Union_{s \ge r} \stratum_{n,s}$
of alternating matrices of rank {\em at most} $n-r$:
this is closed since it is cut out by the vanishing of
the $(n-r+1) \times (n-r+1)$ minors,
and it is in the closure of $\stratum_{n,r}$, as one can see
from using standard symplectic matrices.
We may extend $\overline{\stratum}_{n,r}$ to a closed subscheme
of $\Aff^{n(n-1)/2}_{\Z_p}$ defined by the same equations.

Let $\Zpstratum_{n,r} = \stratum_{n,r}(\Q_p) \intersect M_n(\Z_p)_{\alt}$,
so we have an analogous locally closed stratification of topological spaces
\[
	M_n(\Z_p)_{\alt} = \Union_r \Zpstratum_{n,r}.
\]
The closure $\overline{\Zpstratum}_{n,r}$ of $\Zpstratum_{n,r}$
equals $\overline{\stratum}_{n,r}(\Z_p) = \Union_{s \ge r} \Zpstratum_{n,s}$.

Fix $r$.
Proposition~\ref{P:measure} and 
Corollary~\ref{C:prob measure} applied to $X=\overline{\stratum}_{n,r}$
yields measures $\mu$ and $\nu$ on $\overline{\Zpstratum}_{n,r}$.
By Proposition~\ref{P:measure}\eqref{I:lower dimensional},
$\mu(\overline{\Zpstratum}_{n,s})=0$ for $s>r$,
so the probability measure $\nu$ restricts to a probability measure $\nu$
on the open subset $\Zpstratum_{n,r}$.
We use $\mu$ to denote the $\mu$ for different varieties;
the meaning will be clear from context.

The following generalization of Theorem~\ref{T:u=0}
states that the limit $\Adist_r$ in Theorem~\ref{T:A=T}\eqref{I:Adist exists} 
exists for each $r \in \Z_{\ge 0}$ and gives an explicit formula
for its value:

\begin{theorem}
\label{T:singular distribution}
Fix $r \in \Z_{\ge 0}$,
and fix a symplectic $p$-group $G$.
If $A \in \Zpstratum_{n,r}$ is chosen at random with respect to $\nu$, then
\[
	\lim_{\substack{n \to \infty \\ n-r \textup{ even}}} \Prob\left( (\coker A)_{\tors} \isom G \right)
	= \frac{(\# G)^{1-r}}{\#\Sp(G)} \prod_{i=r+1}^\infty (1-p^{1-2i}).
\]
Moreover, the sum of the right side over all such $G$ equals $1$.
\end{theorem}

To prove Theorem~\ref{T:singular distribution},
we need the following two lemmas.
Let $|\det|^s\colon M_n(\Z_p) \to \R$ 
be the function $A \mapsto |\det A|_p^s$.

\begin{lemma}
\label{L:Igusa}
For any $s \in \R_{\ge 0}$, we have 
$\displaystyle \int_{M_n(\Z_p)_{\alt}} |\det|^s \, \mu 
= \prod_{i=1}^{n/2} \frac{1-p^{1-2i}}{1-p^{1-2i-2s}}$.
\end{lemma}

\begin{proof}
The proof is an easy induction on $n$: see~\cite{Igusa2000}*{p.~164}.
\end{proof}

\begin{lemma}
\label{L:pushforward}
Define
\begin{align*}
  \beta \colon \GL_n(\Z_p) \times \Zpstratum_{n-r,0} &\To \Zpstratum_{n,r} \\
	(M,A) &\longmapsto M^t \begin{pmatrix} A & 0 \\ 0 & 0 \end{pmatrix} M.
\end{align*}
Then $\beta_*(\mu \times |\det|^r \mu) = c \mu$ for some $c>0$
depending on $n$ and $r$.
\end{lemma}

\begin{proof}
Given $B \in M_n(\Z/p^e\Z)$ in the reduction of $\Zpstratum_{n,r}$,
we must count the number of 
$(M,A) \in \GL_n(\Z/p^e\Z) \times M_{n-r}(\Z/p^e\Z)_{\alt}$
such that 
$M^t \begin{pmatrix} A & 0 \\ 0 & 0 \end{pmatrix} M = B$.
We may assume that $e$ is large enough that some $(n-r) \times (n-r)$ minor
of $B$ has nonzero determinant mod $p^e$.
We may assume also that $B$ itself is of the form 
$\begin{pmatrix} C & 0 \\ 0 & 0 \end{pmatrix}$;
then the set of $(M,A)$ is 
\[
	\left\{(N^{-1},N^t B N) : N \in \GL_n(\Z/p^e\Z), N^t B N \textup{ has the form } \begin{pmatrix} * & 0 \\ 0 & 0 \end{pmatrix} \right\}.
\]
If $N =
\begin{pmatrix}
  P & Q \\ R & S
\end{pmatrix}$,
then the condition on $N$ is equivalent to 
$P \in \GL_{n-r}(\Z/p^e\Z)$, $P^t C Q = 0$, $Q^t C P=0$, and $Q^t C Q=0$
(invertibility of $P$ follows from comparing
determinants of minors of $B$ to those of $N^t B N$).
Since $C^t=-C$, these conditions are equivalent to $P \in \GL_{n-r}(\Z/p^e\Z)$
and $CQ=0$.
The number of possibilities for $P$ 
is $\#\GL_{n-r}(\Z/p^e\Z)$,
which is independent of $B$.
On the other hand,
if we view $C$ as a map from the finite group $(\Z/p^e\Z)^{n-r}$ to itself,
its kernel has the same size as its cokernel, which is $|\det C|_p^{-1}$,
so the number of possibilities for $Q$ is $|\det C|_p^{-r}$.
Thus if each pair $(M,A)$ is weighted by $|\det A|_p^r=|\det C|_p^r$,
then the weighted count of such pairs is independent of $B$.
\end{proof}

\begin{proof}[Proof of Theorem~\ref{T:singular distribution}]
Define 
\[
	\Zpstratum_{n,r}(G)\colonequals \{A \in \Zpstratum_{n,r} : 
			(\coker A)_{\tors} \isom G \}.
\]
As $n \to \infty$ through integers with $n-r$ even,
\begin{align*}
  \nu(\Zpstratum_{n,r}(G)) 
	&= \frac{\int_{\Zpstratum_{n,r}(G)} \mu}
		{\int_{\Zpstratum_{n,r}} \mu} \\
	&= \frac{\int_{\GL_n(\Z_p)} \mu \cdot \int_{\Zpstratum_{n-r,0}(G)} |\det|^r \mu}
		{\int_{\GL_n(\Z_p)} \mu \cdot \int_{\Zpstratum_{n-r,0}} |\det|^r \mu} 
		\qquad\textup{(by Lemma~\ref{L:pushforward})} \\
	&= \frac{\#G^{-r} \int_{\Zpstratum_{n-r,0}(G)} \mu}
		{\int_{\Zpstratum_{n-r,0}} |\det|^r \mu} \\
	&\to \frac{\#G^{-r} \frac{\# G}{\#\Sp(G)} \prod_{i=1}^\infty (1-p^{1-2i})}
		{\prod_{i=1}^{\infty} \frac{1-p^{1-2i}}{1-p^{1-2i-2r}}} 
                \qquad\textup{(by Theorem~\ref{T:u=0} and Lemma~\ref{L:Igusa})} \\
	&= \frac{\#G^{1-r}}{\#\Sp(G)} \prod_{i=r+1}^\infty (1-p^{1-2i}).
\end{align*}
The same argument as in the proof of Theorem~\ref{T:u=0}
shows that these numbers sum to $1$.
\end{proof}

\section{Orthogonal Grassmannians}
\label{S:orthogonal Grassmannian}

\subsection{Grassmannians}

Given $0 \le m \le n$,
for each commutative ring $A$ 
let $\Gr_{m,n}(A)$ be the set of direct summands $W$ of $A^n$
that are locally free of rank $m$.
As is well known, this functor is represented by 
a smooth projective scheme $\Gr_{m,n}$ of relative dimension $m(n-m)$ over $\Z$,
called a \defi{Grassmannian}.

\subsection{Maximal isotropic direct summands}
\label{S:maximal isotropic direct summands}

Now equip $A^{2n}$ with the 
hyperbolic quadratic form $Q \colon A^{2n} \to A$
given by
\[
	Q(x_1,\ldots,x_n,y_1,\ldots,y_n) \colonequals \sum_{i=1}^n x_i y_i.
\]
The associated bilinear pairing is 
$\langle a,b \rangle \colonequals Q(a+b)-Q(a)-Q(b)$.
A direct summand $Z$ is called \defi{isotropic} if $Q|_Z=0$
(in general this is stronger than requiring that 
$\langle \;,\; \rangle|_{Z \times Z} = 0$).
Let $\OGr_n(A)$ be the set of isotropic $Z \in \Gr_{n,2n}(A)$.
Such $Z$ will also be called \defi{maximal isotropic direct summands} 
of $A^{2n}$.
Let $W$ be the maximal isotropic direct summand $\Z^n \times 0$ of $\Z^{2n}$.

\begin{lemma}
\label{L:X and X'}
Let $A$ be a ring.
Let $X,X' \in \OGr_n(A)$ be such that $X \directsum X' \to A^{2n}$
is an isomorphism.
\begin{enumerate}[\upshape (a)]
\item \label{I:X' is X^T}
The restriction of $\langle \;,\; \rangle$ to $X \times X'$
identifies $X'$ with $X^T$.
\item \label{I:graph of alternating}
Let $\phi \colon X \to X'$ be an $A$-module homomorphism.
Then $\Graph(\phi) \in \OGr_n(A)$ 
if and only if $\phi$ is alternating 
(with respect to the identification above).
\end{enumerate}
\end{lemma}

\begin{proof}
\hfill
\begin{enumerate}[\upshape (a)]
\item
By tensoring with $A/\mm$ for every maximal ideal $\mm \subseteq A$,
we reduce to the case in which $A$ is a field.
The kernel of $X' \to X^T$ is orthogonal to $X$, but also to $X'$
since $X'$ is isotropic.
By nondegeneracy of $\langle\;,\;\rangle$ on $A^{2n}$, this kernel is $0$.
Since $X'$ and $X^T$ are vector spaces of the same dimension, 
$X' \to X^T$ is an isomorphism.
\item 
For $x \in X$, 
\[
	\langle x, \phi(x) \rangle = Q(x +\phi(x)) - Q(x) - Q(\phi(x)) 
	= Q(x+\phi(x)).
\]
By definition, $\phi$ is alternating if and only if the left side is $0$
for all $x$.
Also by definition, 
$\Graph(\phi) \in \OGr_n(A)$ if and only if the right side is $0$
for all $x \in X$.\qedhere
\end{enumerate}
\end{proof}

\begin{proposition}
\label{P:orthogonal group acts transitively}
Let $\Orthogonal_{2n}$ be the orthogonal group of $(\Z^{2n},Q)$.
\begin{enumerate}[\upshape (a)]
\item \label{I:transitive on 1}
Let $A$ be a field, a discrete valuation ring, or a quotient thereof.
The action of $\Orthogonal_{2n}(A)$ on $\OGr_n(A)$ is transitive.
\item \label{I:transitive on 2}
Let $k$ be a field.
For each $m \in \{0,1,\ldots,n\}$,
the action of $\Orthogonal_{2n}(k)$ on 
$\{ (Y,Z) \in \OGr_n(k)^2 : \dim(Y \intersect Z)=m \}$
is transitive.
\end{enumerate}
\end{proposition}

\begin{proof}
\hfill
\begin{enumerate}[\upshape (a)]
\item
The hypothesis on $A$ implies that every direct summand of $A^{2n}$ is free.
Let $Z \in \OGr_n(A)$.
Choose a basis $z_1,\ldots,z_n$ of $Z$.
Choose a basis $y_1,\ldots,y_n$ 
for an $A$-module complement $Y$ of $Z$ in $A^{2n}$.
Since $\langle \;,\; \rangle$ is nondegenerate,
we can change the basis of $Y$ to assume that
$\langle y_i,z_j \rangle = \delta_{ij}$.
Let $y'_i \colonequals y_i - Q(y_i)z_i - \sum_{j>i} \langle y_i, y_j\rangle z_j$.
Then the $A$-linear map sending
the standard basis of $A^{2n}$ to $z_1,\ldots,z_n,y'_1,\ldots,y'_n$ 
is an element of $\Orthogonal_{2n}(A)$ sending $W$ to $Z$.
\item 
Given $(Y,Z)$ in the set, 
choose a basis $x_1,\ldots,x_m$ for $Y \intersect Z$,
extend it to bases $x_1,\ldots,x_m,z_{m+1},\ldots,z_n$ of $Z$
and $x_1,\ldots,x_m,y_{m+1},\ldots,y_n$ of $Y$,
and replace $y_{m+1},\ldots,y_n$ by linear combinations
so that $\langle y_i, z_j \rangle = \delta_{ij}$ for $i,j \in [m+1,n]$.
Inductively choose $w_i \in k^{2n}$ for $i=1,\ldots,m$
so that $w_i$ is orthogonal to the $w_j$ for $j<i$
and to all the $x_j,y_j,z_j$ except $\langle w_i,x_i \rangle = 1$.
Adjust each $w_i$ by a multiple of $x_i$ in order to assume
in addition that $Q(w_i)=0$.

Now, given another pair $(Y',Z')$ in the set,
the $A$-linear map sending the $w_i,x_i,y_i,z_i$
to their counterparts
is an element of $\Orthogonal_{2n}(A)$ sending $(Y,Z)$ to $(Y',Z')$.\qedhere
\end{enumerate}
\end{proof}

\begin{lemma}
\label{L:ugly}
Let $\Wbar \in \OGr_n(\F_p)$ be the mod $p$ reduction of $W$.
Let $Y \le \Wbar$ be an $\F_p$-subspace.
Then the subgroup of $\Orthogonal_{2n}(\Z_p)$ preserving $W$ and $Y$
acts transitively on $\{X \in \OGr_n(\F_p) : X \intersect \Wbar = Y\}$.
\end{lemma}

\begin{proof}
For any $X,X'$ in the set,
Proposition~\ref{P:orthogonal group acts transitively}\eqref{I:transitive on 2}
yields an element $\bar{\alpha} \in \Orthogonal_{2n}(\F_p)$
sending $(\Wbar,X)$ to $(\Wbar,X')$.
It remains to lift $\bar{\alpha} \in \Stab_{\Orthogonal_{2n}(\F_p)}(\Wbar)$
to an element $\alpha \in \Stab_{\Orthogonal_{2n}(\Z_p)}(W)$,
since such an $\alpha$ will preserve also 
$X \intersect \Wbar = X' \intersect \Wbar = Y$.
By Hensel's lemma, the lift exists if the group scheme
stabilizer $S \le \Orthogonal_{2n}$ of $W$ is smooth over $\Z$.
In fact, if we define $W' \colonequals 0 \times \Z^n \in \OGr_n(\Z)$,
then there is a short exact sequence of group schemes
\[
	1 \to B \to S \to \GL_W \to 1
\]
where $B$ is the additive group scheme of alternating maps
$\beta \colon W' \to W$;
namely, $\beta \in B$ maps to the unique $s \in S$
such that $s(w') = w' + \beta(w')$ for all $w' \in W'$,
and $S \to \GL_W$ is defined by the action of $S$ on $W$.
Since $B$ and $\GL_W$ are smooth, so is $S$.
\end{proof}

\subsection{Orthogonal Grassmannians}

\begin{proposition}
\label{P:OGr is smooth}
For each $n \ge 0$, the functor $\OGr_n$ is represented by
a smooth projective scheme of relative dimension $n(n-1)/2$ over $\Z$,
called an \defi{orthogonal Grassmannian}.
\end{proposition}

\begin{proof}
See \cite{SGA7.2}*{XII, Proposition~2.8},
where $\OGr_n$ is denoted $\textup{G\'en}(X)$.
The expression for the relative dimension arises in the proof there
as the rank of $\bigwedge^2 W$.
\end{proof}

If $V \isom A^{2n}$ for some ring $A$,
also write $\OGr_V$ for the $A$-scheme $\OGr_{n,A}$.

\begin{proposition}
\label{P:OGr+}
Fix $n>0$.
\begin{enumerate}[\upshape (a)]
\item \label{I:two components}
The scheme $\OGr_n$ is a disjoint union of 
two isomorphic schemes $\OGr_n^{\even}$ and $\OGr_n^{\odd}$,
distinguished by the property that for $Z \in \OGr_n(k)$ for a field $k$,
\begin{equation}
  \label{E:OGr_n^even}
	Z \in \OGr_n^{\even}(k) \iff \textup{$\dim(Z \intersect W_k)$ is even}.
\end{equation}
\item 
If $k$ is a field, 
then $\OGr_{n,k}^{\even}$ and $\OGr_{n,k}^{\odd}$ are geometrically integral.
\item \label{I:Z intersect Z'}
For any field $k$,
two points $Z,Z' \in \OGr_n(k)$ 
belong to the same component of $\OGr_{n,k}$ 
if and only if $\dim(Z \intersect Z') \equiv n \pmod{2}$.
\end{enumerate}
\end{proposition}

\begin{proof}
See the proof of \cite{SGA7.2}*{XII, Proposition~2.8},
which shows that there is a morphism $e \colon \OGr_n \to \Spec(\Z \times \Z)$
with geometrically connected fibers.
Define $\OGr_n^{\even}$ and $\OGr_n^{\odd}$ as the preimages
of the components of $\Spec(\Z \times \Z)$;
they can be chosen so that 
\eqref{E:OGr_n^even} and \eqref{I:Z intersect Z'} hold,
by \cite{SGA7.2}*{XII, Proposition~1.12}.
Geometrically connected and smooth imply geometrically integral.
\end{proof}

\begin{remark}
For $n=0$, we may define $\OGr_0^{\even} \colonequals \OGr_0$
and $\OGr_0^{\odd} \colonequals \emptyset$.
\end{remark}

\begin{corollary}
\label{C:sum of three intersection dimensions}
If $Z_1,Z_2,Z_3 \in \OGr_n(k)$ for a field $k$,
then
\[
	  \dim(Z_1 \intersect Z_2) 
	+ \dim(Z_2 \intersect Z_3) 
	+ \dim(Z_3 \intersect Z_1) 
	\equiv n \pmod{2}.
\]
\end{corollary}

\begin{proof}
By Proposition~\ref{P:OGr+}\eqref{I:Z intersect Z'},
the parity of $\dim(Z_1 \intersect Z_2)-n$ measures whether 
$Z_1$ and $Z_2$ belong to the same component.
Summing three such integers gives the parity of the number of 
component switches in hopping from $Z_1$ to $Z_2$ to $Z_3$ and
back to $Z_1$; the latter number is even.
\end{proof}

\begin{lemma}
\label{L:J_V}
Let $q=p^e$ for a prime $p$ and $e \ge 1$.
Then 
\[
	\#\OGr_n(\Z/q\Z) = q^{n(n-1)/2} \prod_{i=1}^n (1+p^{i-n}).
\]
\end{lemma}

\begin{proof}
The case $e=1$ is \cite{Poonen-Rains2012-selmer}*{Proposition~2.6(b)}.
The $e=1$ case implies the general case
since $\OGr_n$ is smooth of relative dimension $n(n-1)/2$.
\end{proof}

\subsection{Schubert subschemes}

Suppose that $0 \le r \le n$.
For a field $k$, let $\Schubert_{n,r}(k)$
be the set of $Z \in \OGr_n^{\parity(r)}(k)$
such that $\dim(Z \intersect W_k) \ge r$,
or equivalently, 
the set of $Z \in \OGr_n(k)$
such that $\dim(Z \intersect W_k) - r \in 2\Z_{\ge 0}$.
For an arbitrary ring $R$,
let $\Schubert_{n,r}(R)$ be the set of $Z \in \OGr_n(R)$
such that $Z_k \in \Schubert_{n,r}(k)$
for every field $k$ that is a quotient of $R$.

\begin{proposition}
\label{P:Schubert}
For $0 \le r \le n$, the functor $\Schubert_{n,r}$ is represented by
a closed subscheme of $\OGr_n$
of relative dimension $n(n-1)/2 - r(r-1)/2 = (n-r)(n+r-1)/2$ over $\Z$.
\end{proposition}

\begin{proof}
There is a closed subscheme of $\Gr_{n,2n}$ 
whose $k$-points parametrize $n$-dimensional subspaces $Z$
with $\dim(Z \intersect W) \ge r$.
Its intersection with the closed subscheme $\OGr_n^{\parity(r)}$ is 
$\Schubert_{n,r}$.

To compute the relative dimension, we work over a field $k$,
and consider the closed subscheme
$\Schubert'_{n,r} \subseteq \Schubert_{n,r} \times \Gr_{r,n}$
parametrizing pairs $(Z,X)$ such that $X \subseteq Z \intersect W$.
Given $X \subseteq W$, 
the quadratic form $Q$ restricts to a hyperbolic quadratic form
on $X^\perp/X$,
and the $Z$'s containing $X$ are in bijection with 
the maximal isotropic subspaces of $X^\perp/X$,
via $Z \mapsto Z/X$.
Thus the second projection $\Schubert'_{n,r} \to \Gr_{r,n}$
has fibers isomorphic to $\OGr_{n-r}$,
so 
\[
	\dim \Schubert'_{n,r} 
	= \dim \Gr_{r,n} + \dim \OGr_{n-r} 
	= r(n-r) + (n-r)(n-r-1)/2 
	= (n-r)(n+r-1)/2.
\]

On the other hand, 
there is an open subscheme $\Schubert_{n,r}^\circ \subseteq \Schubert_{n,r}$
above which $\Schubert_{n,r}' \to \Schubert_{n,r}$ is an isomorphism,
namely the subscheme parametrizing $Z$ for which $\dim(Z \intersect W)$
\emph{equals} $r$.
If we view $\Schubert_{n,r}^\circ$ as an open subscheme of $\Schubert'_{n,r}$,
which maps to $\Gr_{r,n}$,
then its fiber above $X$ is the open subscheme of $\OGr_{X^\perp/X}$
consisting of subspaces $Y$ not meeting $W/X$,
and those subspaces are exactly the graphs of alternating maps
from an $(n-r)$-dimensional space to its dual,
so the fiber is $\Aff^{(n-r)(n-r-1)/2}$.
It follows that $\Schubert_{n,r}^\circ$ has the same dimension 
as $\Schubert'_{n,r}$.
Since $\Schubert_{n,r}$ is sandwiched in between,
it too has the same dimension.
\end{proof}

Call $\Schubert_{n,r}$ a \defi{Schubert subscheme}.
It could also have been defined as the closure of 
the locally closed subscheme $\Schubert_{n,r}^\circ \subseteq \OGr_n$.

\section{Modeling Selmer groups using maximal isotropic submodules}

\subsection{Properties of the short exact sequence}
\label{S:RST}

Let $Z$ and $W$ be maximal isotropic direct summands of $V$ 
as in Section~\ref{S:intersection}.
(For the time being, they do not need to be random;
what we say here applies to any choice of $Z$ and $W$.)
{}From $Z$ and $W$ construct 
\[
	0 \to R \to S \to T \to 0
\]
as in Section~\ref{S:intersection}.

\begin{proposition}
\label{P:maximal divisible subgroup}
The maximal divisible subgroup of $S$ is $R$.
\end{proposition}

\begin{proof}
Since the group 
$R = (Z \intersect W) \tensor \frac{\Q_p}{\Z_p}$ is divisible,
it suffices to show that every infinitely divisible element $a$ of
$S$ is in $R$.
Suppose that $a \in S$ is infinitely divisible.
For each $m \ge 1$, write $a = p^m a_m$ for some $a_m \in S$.
By definition of $S$, we have $a_m = (z_m \bmod V) = (w_m \bmod V)$
for some $z_m \in Z \tensor \Q_p$ and $w_m \in W \tensor \Q_p$.
Choose $n$ such that $a \in p^{-n} V$;
then all the $p^m z_m$ and $p^m w_m$ lie in $p^{-n} V$,
which is compact,
so there is an infinite subsequence of $m$ 
such that the $p^m z_m$ converge and the $p^m w_m$ converge.
The limits must be equal, since $p^m z_m - p^m w_m \in p^m V$.
The common limit in $(Z \intersect W) \tensor \Q_p$
represents $a$.
\end{proof}

\begin{corollary}
\label{C:finiteness of Sha}
The group $T$ is finite.
\end{corollary}

\begin{proof}
The maximal divisible subgroup of a co-finite-type $\Z_p$-module 
is of finite index.
\end{proof}

\begin{corollary}
\label{C:splits}
The exact sequence $0 \to R \to S \to T \to 0$ splits.
\end{corollary}

\begin{proof}
This follows since $R$ is divisible.
\end{proof}

\begin{proposition}
\label{P:S[q]}
If $q$ is a power of $p$,
then $S[q]$ is isomorphic to the intersection $Z/qZ \intersect W/qW$ in $V/qV$.
\end{proposition}

\begin{proof}
Intersecting
\[
	S = \left( Z \tensor \frac{\Q_p}{\Z_p} \right)
		\intersect \left( W \tensor \frac{\Q_p}{\Z_p} \right)
\]
with the $q$-torsion subgroup $\frac{1}{q} V/V$ of $V \tensor \frac{\Q_p}{\Z_p}$
yields
\[
	S[q] = \frac{\frac{1}{q} Z}{Z} \intersect \frac{\frac{1}{q} W}{W}.
\]
The multiplication by $q$ isomorphism $\frac{1}{q} V/V \to V/qV$
sends this to $Z/qZ \intersect W/qZ$.
\end{proof}

\subsection{Model for the Cassels--Tate pairing}
\label{S:model for Cassels-Tate}

Here we define a natural nondegenerate alternating pairing on $T$.
Extend $Q$ to a quadratic form $V \tensor \Q_p \to \Q_p$
and define 
$\langle \;,\; \rangle \colon (V \tensor \Q_p) \times (V \tensor \Q_p) \to \Q_p$
by $\langle x,y \rangle \colonequals Q(x+y)-Q(x)-Q(y)$.
Then $\langle \;,\; \rangle \bmod \Z_p$ identifies $V \tensor \Q_p$
with its own Pontryagin dual, and the subgroup
\[
	Z^\perp \colonequals \{ v \in V \tensor \Q_p : 
	\langle v,z \rangle \bmod \Z_p = 0 \textup{ for all $z \in Z$} \}
\]
equals $Z \tensor \Q_p + V$.
Similarly, $W^\perp = W \tensor \Q_p + V$.

Suppose that $x,y \in T$.
Lift $x$ to $\widetilde{x} \in \left( Z \tensor \frac{\Q_p}{\Z_p} \right)
		\intersect \left( W \tensor \frac{\Q_p}{\Z_p} \right)$.
Choose $z_x \in Z \tensor \Q_p$ whose image in 
$V \tensor \frac{\Q_p}{\Z_p}$ equals $\widetilde{x}$.
Define $w_x$, $\widetilde{y}$, $z_y$, and $w_y$ analogously.

\begin{proposition}
The map
\begin{align*}
  [\;,\;] \colon T \times T &\to \frac{\Q_p}{\Z_p} \\
	x,y & \mapsto Q(z_x-w_y) \bmod \Z_p
\end{align*}
is well-defined, and it is a nondegenerate alternating bilinear pairing.
\end{proposition}

\begin{proof}
First, 
\[
	z_x-w_x \in 
	\ker\left(V \tensor \Q_p \to V \tensor \frac{\Q_p}{\Z_p} \right) 
	= V.
\]
Since $Z$ and $W$ are isotropic, 
\[
	Q(z_x-w_y) = -\langle z_x,w_y \rangle = -\langle z_x-w_x,w_y \rangle,
\]
so changing $w_y$ (by an element of $W$)
changes $Q(z_x-w_y)$ by an element of $\langle V,W \rangle \subseteq \Z_p$,
so $Q(z_x-w_y) \bmod \Z_p$ is unchanged.
Similarly, changing $z_x$ (by an element of $Z$)
does not change $Q(z_x-w_y) \bmod \Z_p$.
If $\widetilde{x}=\widetilde{y}$, then we may choose $w_y=w_x$,
so $Q(z_x-w_y) = Q(z_x-w_x) \in Q(V) \subseteq \Z_p$,
so $Q(z_x-w_y) \bmod \Z_p = 0$.
Thus we have an alternating bilinear pairing on
$\left( Z \tensor \frac{\Q_p}{\Z_p} \right)
		\intersect \left( W \tensor \frac{\Q_p}{\Z_p} \right)$
and it remains to show that the kernel on either side is
$(Z \intersect W) \tensor \frac{\Q_p}{\Z_p}$
so that it induces a nondegenerate alternating pairing on $T$.

The following are equivalent for 
$\widetilde{x} \in \left( Z \tensor \frac{\Q_p}{\Z_p} \right)
		\intersect \left( W \tensor \frac{\Q_p}{\Z_p} \right)$:
\begin{itemize}
\item $Q(z_x-w_y) \bmod \Z_p = 0$ 
for all $\widetilde{y} \in \left( Z \tensor \frac{\Q_p}{\Z_p} \right)
		\intersect \left( W \tensor \frac{\Q_p}{\Z_p} \right)$,
\item $\langle z_x-w_x,w_y \rangle \in \Z_p$,
for all $\widetilde{y} \in \left( Z \tensor \frac{\Q_p}{\Z_p} \right)
		\intersect \left( W \tensor \frac{\Q_p}{\Z_p} \right)$,
\item $\langle z_x-w_x,w \rangle \in \Z_p$,
for all $w \in (Z \tensor \Q_p + V) \intersect (W \tensor \Q_p + V) = Z^\perp \intersect W^\perp$,
\item $z_x-w_x \in (Z^\perp \intersect W^\perp)^\perp = Z + W$, and
\item $\widetilde{x} \in (Z \intersect W) \tensor \frac{\Q_p}{\Z_p}$.\qedhere
\end{itemize}
\end{proof}

\subsection{Predictions for rank}
\label{S:rank}

Corollary~\ref{C:prob measure} defines a probability measure on $\OGr_n(\Z_p)$.
In the definition of $\Qdist_{2n}$ 
we chose both $Z$ and $W$ randomly from $\OGr_n(\Z_p)$.
But since the orthogonal group of $(V,Q)$ acts transitively on $\OGr_V(\Z_p)$,
fixing $W$ to be $\Z_p^n \times 0$ 
as in Section~\ref{S:maximal isotropic direct summands}
and choosing only $Z$ at random
would produce the same distribution.
A similar comment applies to $\Tdist_{2n,r}$.
{}From now on, we assume that $W$ is fixed as above.

\begin{proposition}
\label{P:rank 0 or 1}
Fix $n$.
If $Z$ is chosen randomly from $\OGr_n(\Z_p)$,
then the $\Z_p$-module $Z \intersect W$ is free of rank $0$ or $1$,
with probability $1/2$ each.
\end{proposition}

\begin{proof}
By Proposition~\ref{P:Schubert}, $\dim \Schubert_{n,r} < \dim \OGr_n$
for $r \ge 2$,
so the probability that $\rk(Z \intersect W) \ge 2$ is $0$
by Proposition~\ref{P:measure}\eqref{I:lower dimensional}.
On the other hand, $\OGr_n^{\even}$ and $\OGr_n^{\odd}$ are isomorphic 
by Proposition~\ref{P:OGr+}\eqref{I:two components},
so the parity of $\rk(Z \intersect W)$ is equidistributed.
\end{proof}

Conjecture~\ref{C:main} implies that the distribution 
of $E(k) \tensor \frac{\Q_p}{\Z_p}$
matches that of $(Z \intersect W) \tensor \frac{\Q_p}{\Z_p}$,
or equivalently that 
the distribution of $E(k) \tensor \Z_p$
matches that of $Z \intersect W$.
Thus it implies that 50\% of elliptic curves
over $k$ have rank $0$, and 50\% have rank $1$.

\subsection{Random models and their compatibility}
\label{S:comparison}

One can show that the locus of $Z \in \OGr_n(\Z_p)$ for which 
the sequence $0 \to R \to S \to T \to 0$
is isomorphic to a given sequence
is locally closed in the $p$-adic topology,
and hence measurable.
This would show that $\Qdist_{2n}$ is well-defined.
A similar argument using the probability measure on $\Schubert_{n,r}(\Z_p)$
would show that $\Tdist_{2n,r}$ is well-defined.
We find it more convenient, however, to prove these measurability claims
and to prove that $\lim_{n \to \infty} \Qdist_{2n}$ 
and $\lim_{n \to \infty} \Tdist_{2n,r}$
exist by relating them to the distributions $\Adist_{n,r}$.
Recall that we 
already proved 
in Section~\ref{S:singular distribution}
that the $\Adist_{n,r}$ exist and converge 
to a limit $\Adist_r$ as $n \to \infty$ through integers
with $n-r \in 2\Z_{\ge 0}$.

Before giving the proof that 
the limit $\lim_{n \to \infty} \Tdist_{2n,r} \equalscolon \Tdist_r$ exists
and coincides with $\Adist_r$ (Theorem~\ref{T:A=T}\eqref{I:Adist=Tdist}),
let us explain the idea.
There is a simple relationship between alternating matrices $A$
and maximal isotropic direct summands $Z$: 
namely, if we view $A$ as a linear map $W \to W^T$, 
then $Z \colonequals \Graph(A) \subset W \directsum W^T$ is maximal isotropic.
But not every maximal isotropic direct summand $Z \le W \directsum W^T$
comes from an $A$.
Over a field, the $Z$'s that arise are 
those that intersect $W^T$ trivially; at the other extreme is $W^T$ itself;
a general $Z$ is a hybrid of these two extremes:
namely, they arise by writing $W = W_1 \directsum W_2$,
forming the corresponding decomposition $W^T = W_1^T \directsum W_2^T$,
and taking $Z \colonequals W_1^T \directsum \Graph(A)$ 
for some alternating $A \colon W_2 \to W_2^T$.
We need to work over $\Z_p$ instead of a field,
but we can still represent 
a general $Z$ in terms of a decomposition as above
(note, however, that the $Z$'s that arise directly from an $A$
on the whole of $W$ 
are those for which the \emph{mod $p$ reductions} of $Z$ and $W^T$
intersect trivially).
Moreover, we will show that the uniform distribution of $Z \in \Schubert_{n,r}$
can be obtained by choosing the decompositions of $W$ and $W^T$
at random (with respect to a suitable measure) 
and then choosing $A \colon W_2 \to W_2^T$
at random from those alternating maps whose kernel has rank $r$.
The distribution $\Tdist_{2n,r}$ is defined 
in terms of the group $T$ arising from $Z$.
It turns out that $T \isom (\coker A)_{\tors}$,
and one shows that with high probability as $n \to \infty$, 
the size of $A$ is large, 
so the distribution of $(\coker A)_{\tors}$ 
is well approximated by $\Adist_r$.

\begin{proof}[Proofs of Theorems \ref{T:T distribution} and~\ref{T:A=T}\eqref{I:Adist=Tdist}]
We use the notation $\Vbar \colonequals V/pV$.
Fix a maximal isotropic $\F_p$-subspace $\Lambda \le \Vbar$
with respect to the $\F_p$-valued quadratic form $(Q \bmod p)$
such that $\Lambda \intersect \Wbar = 0$ in $\Vbar$.
Define a distribution $\GG$ on submodules $Z \le V$ as follows:
\begin{enumerate}[\upshape 1.]
\item 
Choose a maximal isotropic direct summand $W^T \le V$
at random conditioned on $\overline{W^T} = \Lambda$.
(Then $\langle\;,\;\rangle|_{W \times W^T}$ is nondegenerate mod $p$,
so it identifies $W^T$ with the $\Z_p$-dual of $W$,
so the name $W^T$ makes sense.
Also, $V = W \directsum W^T$.)
\item
Choose $m \in \{0,1,\ldots,n-r\}$ 
at random so that its distribution matches
the distribution of $\dim(\overline{\calZ} \intersect \Lambda)$
for $\calZ$ chosen from $\Schubert_{n,r}(\Z_p)$.
The scheme $\Schubert_{n,r}$ is contained in $\OGr_n^{\parity(r)}$,
so $\dim(\overline{\calZ} \intersect \Wbar) \equiv r \pmod{2}$.
Corollary~\ref{C:sum of three intersection dimensions}
applied to $(\calZ,\Wbar,\Lambda)$ implies that $m+r \equiv n \pmod{2}$.
\item
Choose a random $\Z_p$-module decomposition of $W$ as $W_1 \directsum W_2$
such that $\rk W_1 = m$.
Let $W^T = W_1^T \directsum W_2^T$ be the induced decomposition of $W^T$;
i.e., $W_2^T$ is the annihilator of $W_1$ 
with respect to $\langle\;,\;\rangle|_{W \times W^T}$,
and $W_1^T$ is the annihilator of $W_2$.
(Then $W_i^T$ is isomorphic to the $\Z_p$-dual of $W_i$ for $i=1,2$.)
\item
Choose an alternating $\Z_p$-linear map $A \colon W_2 \to W_2^T$
at random from maps whose kernel has rank $r$
(since $\rk W_2 = n-m \equiv r \pmod{2}$, the set of such $A$ is nonempty).
Let $\Graph(A) \le W_2 \times W_2^T$ be its graph.
Let $Z = W_1^T \directsum \Graph(A)$.
\end{enumerate}
Since $A$ is alternating, the direct summand $\Graph(A)$ of $V$ is isotropic.
Since $W_1^T \le W^T$, the direct summand $W_1^T$ is isotropic.
Under $\langle\;,\;\rangle|_{W \times W^T}$, 
the direct summand $W_1^T$ annihilates $W_2^T$ 
(since both are contained in $W^T$)
and $W_2$ (by definition).
The previous three sentences show that $Z$ is an isotropic direct summand.
Its rank is $\rk W_1^T + \rk W_2 = \rk W_1 + \rk W_2 = n$,
so $Z$ is a maximal isotropic direct summand.

Reducing modulo $p$ yields
\[
	\Zbar = \overline{W_1^T} \directsum \Graph(\overline{A}),
\]
so in $\Vbar$ we have
\[
	\Zbar \intersect \Lambda = 
	\Zbar \intersect \overline{W^T} = \overline{W_1^T},
\]
which is of $\F_p$-dimension $m$.

{\bf Claim:} \emph{$\GG$ coincides with the uniform distribution on $\Schubert_{n,r}$.}
Both distributions assign the uniform measure to the set 
of maximal isotropic direct summands $Z$ with $\dim(Z \intersect W)=r$
having a fixed mod~$p$ reduction $\Zbar$,
so it suffices to show that the distributions of $\Zbar$ match.
For each $m$,
both distributions for $\Zbar$ are uniform over all maximal isotropic
subspaces of $\Vbar$ 
for which $\dim(\Zbar \intersect \Wbar) \ge r$
and $\dim(\Zbar \intersect \Lambda) = m$,
so it suffices to prove that the distribution of the integer
$\dim(\Zbar \intersect \Lambda)$ 
is the same for both distributions.
The latter holds by the choice of $m$.
This proves the claim.

For $Z$ sampled from $\GG$, the definition of $Z$ yields
\[
	\left(Z \tensor \frac{\Q_p}{\Z_p} \right) \intersect 
		\left( W \tensor \frac{\Q_p}{\Z_p} \right)
	= \Graph\left(A \tensor \frac{\Q_p}{\Z_p} \right) \intersect \left(W \tensor \frac{\Q_p}{\Z_p} \right)
	\isom \ker\left(A\tensor \frac{\Q_p}{\Z_p}\right),
\]
whose Pontryagin dual is $\coker A$.
The quotient $T$ of the left side
by its maximal divisible subgroup $(Z \intersect W) \tensor \frac{\Q_p}{\Z_p}$
is dual to the finite group $(\coker A)_{\tors}$,
hence isomorphic to $(\coker A)_{\tors}$.
Thus the distribution $\Tdist_{2n,r}$ of $T$ is a weighted average over $m$ 
of the distribution 
$\Adist_{n-m,u}$ of $(\coker A)_{\tors}$ for $A \in \Zpstratum_{n-m,u}$;
this proves in particular that $\Tdist_{2n,r}$ is well-defined.

We next show that as $n \to \infty$, the probability that $m$ is small,
say less than $n/2$, tends to $1$.  
In fact, we show that this holds even after conditioning
on the intersection $\overline{\calZ}\cap \Wbar$;
i.e., we will prove that
\[
	\inf_{Y} 
	\Prob \left(m < n/2 \mid \overline{\calZ}\cap \Wbar=Y \right) \to 1
\]
as $n \to \infty$,
where $Y$ ranges over the possibilities for $\overline{\calZ}\cap \Wbar$.
Fix $Y$.
Let $y \colonequals \dim Y$.
Since $m + y \le \dim \overline{\calZ} = n$, 
the probability is $1$ if $y>n/2$, so assume that $y \le n/2$.
The subgroup of $\Orthogonal_{2n}(\Z_p)$ preserving $W$ and $Y$
acts transitively on the maximal isotropic subspaces $\overline{\calZ}$ 
of $\F_p^{2n}$ satisfying $\overline{\calZ} \intersect \Wbar = Y$,
by Lemma~\ref{L:ugly}, 
so the distribution of $\overline{\calZ}$ is uniform among such subspaces.
Thus $\overline{\calZ}/Y$ is a uniformly random maximal isotropic subspace
of $Y^\perp/Y$ intersecting $\Wbar/Y$ trivially.
In $Y^\perp/Y$, the image of $\Lambda \intersect Y^\perp$ 
is a maximal isotropic complement $C$ of $\Wbar/Y$,
so $\overline{\calZ}/Y$ is the graph of a uniformly random alternating map 
$B \colon C \to \Wbar/Y$ 
(see Lemma~\ref{L:X and X'}\eqref{I:graph of alternating}).
Then $m = \dim \ker B$.
By Lemma~\ref{L:kernel at least n/2},
$m < (n-y)/2$ with high probability,
so $m < n/2$ with high probability.

So the size $n-m$ of the matrix $A$ is large with high probability,
and we have already seen that $n-m \equiv r \pmod{2}$.
Thus the weighted average converges as $n \to \infty$ to $\Adist_r$.
In other words, $\Tdist_r$ exists and coincides with $\Adist_r$.
\end{proof}

We now prove that the distributions $\Qdist_{2n}$ exist 
and converge to $\Qdist$ as $n \to \infty$.

\begin{proof}[Proof of Theorem~\ref{T:limit of RST distribution}]
Define a new distribution $\Qdist'_{2n}$ on short exact sequences as follows.
Choose $r \in \{0,1\}$ uniformly at random, and let $R=(\Q_p/\Z_p)^r$.
Choose $T$ with respect to the distribution $\Tdist_{2n,r}$.
Form the exact sequence
\[
	0 \To R \To R \directsum T \To T \to 0.
\]
By Proposition~\ref{P:rank 0 or 1} and Corollary~\ref{C:splits}, 
the distribution $\Qdist_{2n}$ coincides with $\Qdist'_{2n}$; 
in particular, it is well-defined.

For each $r$, the distribution $\Tdist_{2n,r}$ tends to a limit
as $n \to \infty$,
so the same is true of $\Qdist'_{2n} = \Qdist_{2n}$.
\end{proof}

\subsection{Predictions for \texorpdfstring{$\Sel_{p^e}$}{Sel p e}}
\label{S:p^e}

\begin{lemma}
\label{L:torsion=0}
Fix a global field $k$.
Asymptotically 100\% of elliptic curves over $k$ satisfy $E(k)_{\tors}=0$.
\end{lemma}

\begin{proof}
For each global field $k$ and prime $p$,
the theory of modular curves and Igusa curves
shows that the generic elliptic curve
(over $k(a_1,a_2,a_3,a_4,a_6)$)
has no nonzero rational $p$-torsion point.
By the Hilbert irreducibility theorem,
the same holds for asymptotically 100\% of elliptic curves over $k$.
The size of the torsion subgroup is bounded by a constant 
depending only on $k$ 
\cites{Levin1968,Mazur1977,Kamienny-Mazur1995,Merel1996},
so we need consider only finitely many $p$.
Thus 100\% of $E \in \EE$ satisfy $E(k)_{\tors}=0$.
\end{proof}

\begin{remark}
One could also prove Lemma~\ref{L:torsion=0} without using~\cite{Merel1996}:
the torsion subgroup can also be controlled by 
using reduction modulo primes.
\end{remark}

\begin{proposition}
\label{P:small Selmer, big Selmer}
Suppose that $E$ is an elliptic curve over a global field $k$
with $E(k)_{\tors}=0$.
Let $m$ and $n$ be positive integers such that $\Char k \nmid m,n$ and $m|n$.
Then
\begin{enumerate}[\upshape (a)]
\item The inclusion $E[m] \to E[n]$ induces an isomorphism
$\HH^1(k,E[m]) \to \HH^1(k,E[n])[m]$.
\item\label{I:Sel_m in Sel_n} 
This isomorphism identifies $\Sel_m E$ with $(\Sel_n E)[m]$.
\item \label{I:Sel_q in Sel_p^infty}
If $p$ is a prime number and $e \in \Z_{\ge 0}$, then 
$\Sel_{p^e} E \isom (\Sel_{p^\infty} E)[p^e]$.
\end{enumerate}
\end{proposition}

\begin{proof}\hfill
\begin{enumerate}[\upshape (a)]
\item 
Taking cohomology of $0 \to E[m] \to E[n] \stackrel{m}\to E[n/m] \to 0$
yields a homomorphism $\alpha$ fitting into the exact sequence
\[
	0 \to \HH^1(k,E[m]) \to \HH^1(k,E[n]) \stackrel{\alpha}\to \HH^1(k,E[n/m]).
\]
Replacing $m$ by $n/m$ shows that $E[n/m] \injects E[n]$
induces an injection $\HH^1(k,E[n/m]) \to \HH^1(k,E[n])$.
The composition $E[n] \stackrel{m}\to E[n/m] \injects E[n]$
induces a composition 
$\HH^1(k,E[n]) \stackrel{\alpha}\to \HH^1(k,E[n/m]) \injects \HH^1(k,E[n])$
that equals multiplication by $m$,
so $\HH^1(k,E[m]) \isom \ker \alpha \isom \HH^1(k,E[n])[m]$.
\item 
An element of $\HH^1(k,E[m])$ lies in the subgroup $\Sel_m E$
if and only if its image in $\HH^1(k,E[n])$ lies in $\Sel_n E$,
since the condition for either is 
that it map to $0$ in $\HH^1(k_v,E)$ for every $v$.
\item 
Apply \eqref{I:Sel_m in Sel_n} to $p^e | p^n$ 
and take the direct limit as $n \to \infty$.\qedhere
\end{enumerate}
\end{proof}

Let $q=p^e$ for some prime $p$ and $e \ge 0$.
Because of 
Propositions \ref{P:S[q]}
and~\ref{P:small Selmer, big Selmer}\eqref{I:Sel_q in Sel_p^infty},
Conjecture~\ref{C:main}
implies that the distribution of $\Sel_q E$ 
is the limit as $n \to \infty$
of the distribution
of $Z \intersect W$ for random $Z,W \in \OGr_n(\Z/q\Z)$.
Taking $q=p$, we recover~\cite{Poonen-Rains2012-selmer}*{Conjecture~1.1(a)}.

Given $m \in \Z_{\ge 0}$ and a finite $\Z/q\Z$-module $G$, let $I_m(G)$
be the number of injective homomorphisms $(\Z/q\Z)^m \to G$.
An inclusion-exclusion argument shows that $I_m(G)$
is a monic degree $m$ polynomial in $\#G$ with coefficients in $\Z[q]$.
Thus if $G$ is sampled from some distribution on finite $\Z/q\Z$-modules,
then knowledge of the averages of $I_m(G)$ for all $m \ge 0$
is equivalent to knowledge of all moments of $\#G$.
For the distribution of $Z \intersect W$ for $Z,W \in \OGr_n(\Z/q\Z)$,
it turns out that the formulas for 
the averages of $I_m(G)$ are simpler than the formulas for the moments:

\begin{theorem}
\label{T:free m-tuples in Sel_q}
Fix $m \in \Z_{\ge 0}$.
The average of $I_m(Z \intersect W)$ as $Z,W$ vary over $\OGr_n(\Z/q\Z)$
tends to $q^{m(m+1)/2}$ as $n \to \infty$.
\end{theorem}

\begin{proof}
For each $n$, we may fix $W$.
The desired number is the number
of injective homomorphisms $h \colon (\Z/q\Z)^m \to W$
times the probability that a random $Z \in \OGr_n(\Z/q\Z)$ contains $\im(h)$.
The number of $h$'s is 
$(\#W)^m \prod_{i=0}^{m-1} (1-q^{i-n})$.
The $Z$'s containing $\im(h)$ correspond to the maximal isotropic
direct summands of $\im(h)^\perp/\im(h)$,
a hyperbolic quadratic $\Z/q\Z$-module of rank $2n-2m$,
so their number is $\#\OGr_{n-m}(\Z/q\Z)$.
Using Lemma~\ref{L:J_V},
we compute
\[
	(\#W)^m \prod_{i=0}^{m-1} (1-p^{i-n}) \frac{\#\OGr_{n-m}(\Z/q\Z)}{\#\OGr_n(\Z/q\Z)}
	= q^{m(m+1)/2} \prod_{i=0}^{m-1} (1-p^{i-n}) \prod_{i=n-m}^{n-1} (1+p^{-i}),
\]
which tends to $q^{m(m+1)/2}$ as $n \to \infty$.
\end{proof}

Theorem~\ref{T:free m-tuples in Sel_q} suggests the following:

\begin{conjecture}
\label{C:moments}
For each $m \ge 0$, the average of $I_m(\Sel_q E)$ over $E \in \EE$
exists and equals $q^{m(m+1)/2}$.
\end{conjecture}

\begin{remark}
The combination of Conjecture~\ref{C:main} 
and Theorem~\ref{T:free m-tuples in Sel_q} 
does not quite imply Conjecture~\ref{C:moments},
because it could be that a density~$0$ subset of $\EE$
contributes a positive amount towards the average.
But if we assume also the weak conjecture that
every moment of $\#\Sel_q E$ is bounded (in the $\limsup$ sense),
then the boundedness of the $(m+1)^{\textup{st}}$ moment
implies that no density~$0$ subset of $\EE$
contributes a positive amount towards the $m^{\textup{th}}$ moment,
so the average of $I_m(\Sel_q E)$ for $E \in \EE$
is $q^{m(m+1)/2}$.
\end{remark}

\begin{remark}
The case of Theorem~\ref{T:free m-tuples in Sel_q} 
in which $q$ is a prime $p$
is equivalent to \cite{Poonen-Rains2012-selmer}*{Proposition~2.22(a)},
which states that the $m^{\tH}$ moment of $\#(Z \intersect W)$ 
equals $\prod_{i=1}^m (p^i+1)$.
Theorem~\ref{T:free m-tuples in Sel_q} 
makes it possible to compute the moments also for non-prime $q$,
but the answers appear to be complicated.
See \cite{Delaunay-Jouhet-preprint} for an analogous calculation
of the conjectural moments of $\#\Sha[p^e]$.
\end{remark}

\begin{remark}
For $q=2$ and $m=1$, 
the result of Theorem~\ref{T:free m-tuples in Sel_q}
can be related to the Tamagawa number $\tau(\PGL_2)=2$.
(See \cite{Bhargava-Shankar-preprint1} and \cite{Poonen-bourbaki-preprint}.)
Is there a Tamagawa number explanation for all $q$ and $m$?
\end{remark}

\subsection{Considering all \texorpdfstring{$p$}{p}-primary parts at once}
\label{S:Zhat}

Let $\Sel E \colonequals \varinjlim_n \Sel_n E$
be the direct limit over all $n \in \Z_{>0}$, ordered by divisibility,
so $\Sel E \isom \Directsum_p \Sel_{p^\infty} E$.
It fits in an exact sequence
\[
	0 \To E(k) \tensor \frac{\Q}{\Z} \To \Sel E 
	\To \Sha \To 0
\]
of discrete $\Zhat$-modules (i.e., torsion abelian groups).
The $p$-primary parts of this sequence should not be completely independent,
because if $\Sha$ is finite, then the $\Z_p$-corank of
the $p$-primary part $\Sel_{p^\infty} E$ of $\Sel E$
is independent of $p$.

Therefore we condition on the rank $r$,
in which case we need only focus on the model for $\Sha$.
Here is our model:
independently for each prime $p$,
choose a finite symplectic abelian $p$-group $T_p$ 
with respect to $\Tdist_r$
(or equivalently $\Adist_r$, by Theorem~\ref{T:A=T}\eqref{I:Adist=Tdist}),
and define $T \colonequals \Directsum_p T_p$.

\begin{theorem}
\label{T:Sha for r at least 1}
If $r \ge 1$, then the group $T$ above is finite with probability~$1$, 
and has the distribution of \cite{Delaunay2001}*{Heuristic Assumption}, 
with the correction that $r/2$ is replaced by $r$.
\end{theorem}

\begin{proof}
By Theorem~\ref{T:singular distribution} for $G=0$,
\[
	\Prob(T_p \ne 0) = 1 - \prod_{i=r+1}^\infty (1-p^{1-2i}) = O(p^{-1-2r}).
\]
If $r \ge 1$, then $\sum_p \Prob(T_p \ne 0)$ converges,
so the Borel--Cantelli lemma implies that $T_p=0$ for all but finitely $p$
with probability~$1$,
so $T$ is finite with probability~$1$.
The probability that $T$ is isomorphic to 
a given symplectic abelian group $G$
is the (convergent) product over $p$ 
of the probability that $T_p \isom G[p^\infty]$.
Since the formula in \cite{Delaunay2001}*{Heuristic Assumption}
is multiplicative on $p$-primary parts, the result follows.
\end{proof}

For the rest of this section, assume that $r=0$.
Then $\sum_p \Prob(T_p \ne 0)$ diverges, 
and the probability that
$T$ is isomorphic to any particular finite abelian group is $0$, so we
do not obtain a discrete probability distribution on finite abelian
groups.
This situation is similar to that for class groups of imaginary 
quadratic fields: the density of such fields having 
a specified class group is $0$.
In the class group setting, 
the article \cite{Cohen-Lenstra1983} 
formulated nontrivial statements
by measuring the probability not of individual groups
but of certain infinite sets of isomorphism classes of groups,
and more generally, by computing the average 
of certain functions $f$ defined on such isomorphism classes.
Following \cite{Delaunay2001}, we will do something analogous
for symplectic abelian groups.

Let $\EE_{0,<X}$ be the set of $E \in \EE_0$ of height less than $X$.
We use $\sum_G$ to denote a sum over (isomorphism classes of)
symplectic abelian groups;
we often restrict the sum by imposing conditions on the size of $G$.
For a symplectic abelian group $G$, define
\[
	w_G \colonequals \frac{\#G}{\#\Sp(G)}.
\]
For $k=\Q$, 
Delaunay~\cite{Delaunay2001}*{Heuristic Assumption}, 
inspired by~\cite{Cohen-Lenstra1983},
proposed the heuristic
\begin{equation}
  \label{E:Delaunay}
	\lim_{X \to \infty} \frac{\sum_{E \in \EE_{0,<X}} f(\Sha(E))}
	{\sum_{E \in \EE_{0,<X}} 1}
	\stackrel{?}= 
	\lim_{n \to \infty}
	\frac{\sum_{\#G \le n} f(G) w_G}
	{\sum_{\#G \le n} w_G}.
\end{equation}

Some hypotheses on $f$ are necessary since one can construct 
wildly oscillating functions $f$
for which even the ``easy'' limit on the right side of~\eqref{E:Delaunay}
fails to exist.
Let us now describe a class of functions for which we expect 
equality in~\eqref{E:Delaunay}.
Fix a set of primes $P$ such that $\sum_{p \in P} 1/p < \infty$.
Given $G$,
write $G=H_G \times H'_G$ where $\#H_G$ is divisible only
by primes in $P$, and $\#H'_G$ is divisible only by primes not in $P$.
We use $\sum_H$ (resp.\ $\sum_{H'}$) to denote a sum restricted 
to symplectic abelian groups of order divisible only by primes in $P$
(resp.,\ not in $P$); again we may also impose restrictions 
on the size of $H$ or $H'$.
Then $\sum_{p \in P} \Prob(T_p \ne 0) \le \sum_{p \in P} O(1/p) < \infty$,
so the Borel--Cantelli lemma implies that the random group
$\Directsum_{p \in P} T_p$ 
is given by a discrete probability distribution
on the set of isomorphism classes of (finite) symplectic abelian groups $H$ 
of order divisible only by primes in $P$;
in fact, Theorem~\ref{T:u=0} implies that
$\Prob\left(\Directsum_{p \in P} T_p \isom H \right) = c_P w_H$,
where $c_P$ is a normalizing constant defined as
the convergent product $\prod_{p \in P} \prod_{i=1}^\infty (1-p^{1-2i})$;
in particular, $\sum_H w_H < \infty$.
By an \defi{$L^1$ function} on the set of such $H$,
we mean a real-valued function $f$ such that 
$\sum_H |f(H)| w_H < \infty$;
in particular, bounded functions are $L^1$.
Given such an $L^1$ function $f$, we define 
\[
	\int f \colonequals 
		\frac{\sum_H f(H) w_H}
		{\sum_H w_H}
\]
and extend $f$ to all symplectic abelian groups $G$
by defining $f(G) \colonequals f(H_G)$.
It is reasonable to conjecture~\eqref{E:Delaunay} for $L^1$ functions $f$.
On the other hand, our model suggests the conjecture
that 
\begin{equation}
  \label{E:our f(Sha)}
	\lim_{X \to \infty} \frac{\sum_{E \in \EE_{0,<X}} f(\Sha(E))}
	{\sum_{E \in \EE_{0,<X}} 1}
	\stackrel{?}= 
	\int f
\end{equation}
for such $L^1$ functions $f$.
We now prove that Delaunay's prediction agrees with ours,
i.e., that the right sides of \eqref{E:Delaunay} and \eqref{E:our f(Sha)} 
are equal.

\begin{theorem}
\label{T:Sha for r=0}
Let $P$ be a set of primes such that $\sum_{p \in P} 1/p < \infty$.
Let $f$ be an $L^1$ function 
on the set of (isomorphism classes of) symplectic abelian groups $H$
of order divisible only by primes in $P$.
Extend $f$ to all symplectic abelian groups $G$ by defining
$f(G) \colonequals f(H_G)$.
Then 
\begin{equation}
\label{E:limit of weighted average}
	\lim_{n \to \infty}
	\frac{\sum_{\#G \le n} f(G) w_G}
	{\sum_{\#G \le n} w_G}
	=
	\int f.
\end{equation}
\end{theorem}

Before starting the proof of Theorem~\ref{T:Sha for r=0},
we prove bounds on sums involving $w_G$.

\begin{lemma}
\label{L:1/n 2/n}
For any $N \ge 1$, we have $1/N \le \sum_{\#G = N^2} w_G \le 2/N$.
\end{lemma}

\begin{proof}
The sum in Theorem~\ref{T:singular distribution} being $1$ 
implies that 
\[
	\sum_{k=0}^\infty \sum_{\#G = p^{2k}} w_G t^k 
	= \prod_{i=1}^\infty (1-p^{1-2i} t)^{-1}
\]
holds for $t=p^{-2r}$ for all $r \in \Z_{\ge 0}$, 
so it holds identically in $\Q[[t]]$.
Apply the $q$-binomial theorem to the right side 
(take $x=p^{-2}$ and $z=pt$ in the expression~$Z$ in \cite{Euler1748}*{\S313})
and equate coefficients of $t^k$ to obtain
\[
	\sum_{\#G = p^{2k}} w_G
	= p^{-k} \prod_{j=1}^k (1-p^{-2j})^{-1}
\]
(which is equivalent to \cite{Delaunay2001}*{Corollary~6}).
Take the product over the prime powers in the factorization of 
a positive integer $N$, 
and use
\[
	1 \le \prod_{\textup{prime powers $m>1$ dividing $N$}} (1-m^{-2})^{-1}
	\le \prod_{m=2}^\infty (1-m^{-2})^{-1} 
	= 2.\qedhere
\]
\end{proof}

\begin{corollary}
\label{C:sum w_G}
We have $\sum_{\#G \le n} w_G \ge \frac12 \log n$
and $\sum_{\#G \in [\ell,n]} w_G = O(\log(n/\ell))$.
\end{corollary}

\begin{proof}[Proof of Theorem~\ref{T:Sha for r=0}]
We may add a constant to $f$ in order to assume that $\int f = 0$;
in other words, $\sum_H f(H) w_H = 0$.
For any $M \in \R$, define $S_M \colonequals \sum_{\#H \le M} f(H) w_H$;
thus the $S_M$ are bounded and $\lim_{M \to \infty} S_M = 0$.
Suppose that $\epsilon>0$ is given;
fix $m$ such that $M > m$ implies $|S_M| < \epsilon$.
Then
\begin{align*}
	\left| \sum_{\#G \le n} f(G) w_G \right| 
	&= \left| \sum_{\# H' \le n} \; \sum_{\#H \le \frac{n}{\#H'}} f(H) w_H w_{H'} \right| \qquad\textup{(we write each $G$ as $H \times H'$)}\\
	&\le \sum_{\# H' \le n} w_{H'} \left| S_{n/\#H'} \right| \\
	&\le \sum_{\# H' < n/m} w_{H'} \epsilon + \sum_{\# H' \in [n/m,n]} w_{H'} O(1) \\
	&\le \left( \frac12 \log n \right) \epsilon  + O(\log m) \qquad\textup{(by Corollary~\ref{C:sum w_G})} \\
\intertext{and}
	\sum_{\#G \le n} w_G &\ge \frac12 \log n \qquad\textup{(by Corollary~\ref{C:sum w_G}).}
\end{align*}
Thus the lim sup of the absolute value of 
the ratio in~\eqref{E:limit of weighted average}
is bounded by $\epsilon$.
This holds for every $\epsilon$, so the limit is $0$,
matching $\int f$.
\end{proof}

\section{Arithmetic justification}
\label{S:arithmetic justification}

In this section, we prove results on the arithmetic of elliptic curves 
that partially explain why $\Sel_{p^e} E$
should behave like an intersection of maximal isotropic direct summands.

\subsection{Shafarevich--Tate groups of finite group schemes}

For any $G_k$-module or finite $k$-group scheme $M$, define 
\[
	\Sha^1(k,M) \colonequals 
	\ker\left( \HH^1(k,M) \to \prod_{v \in \Omega} \HH^1(k_v,M) \right).
\]
(If $M$ is not \'etale, then the cohomology 
should be interpreted as fppf cohomology.)

\begin{proposition}
\label{P:Sha^1}
Let $E$ be an elliptic curve over a global field $k$.
Let $p$ be a prime and let $e \in \Z_{\ge 0}$.
If $\Char k \ne p$,
suppose that the image $G$ of $G_k \to \Aut E[p^e] \isom \GL_2(\Z/p^e\Z)$
contains $\SL_2(\Z/p^e\Z)$.
If $\Char k=p$,
suppose that the image $G$ of $G_k \to \Aut E[p^e](\ksep)$ is cyclic.
Then $\Sha^1(k,E[p^e])=0$.
\end{proposition}

\begin{remark}
\label{R:hypothesis holds}
For each $k$, the hypothesis of Proposition~\ref{P:Sha^1}
holds for 100\% of elliptic curves over $k$,
as we now explain.
If $\Char k \ne p$, then the result follows from the Hilbert
irreducibility theorem.
If $\Char k=p$ and either $p>2$ or $e \le 2$,
then $E[p^e](\ksep)$ is cyclic of order $p^f$ for some $f \le e$,
and its automorphism group is $(\Z/p^f\Z)^\times$, which is cyclic;
thus the hypothesis holds for \emph{all} elliptic curves over $k$.
Finally, if $\Char k =2$, then an explicit calculation with Weierstrass
equations shows that $E[2](\ksep)=0$ for 100\% of $E$,
and in that case, $E[2^e](\ksep)=0$ follows.
\end{remark}

Before proving Proposition~\ref{P:Sha^1},
we introduce some more definitions and prove a few basic facts.
For any finite group $G$ and $G$-module $M$, define
\[
	\HH^1_{\cyc}(G,M) \colonequals 
	\Intersection_{\textup{cyclic $H \le G$}} 
	\ker\left( \HH^1(G,M)\to \HH^1(H,M) \right),
\]
which, like $\HH^1(G,M)$, is contravariant in $G$ and covariant in $M$.
For any Galois extension $L/k$ and $\Gal(L/k)$-module $M$, 
define
\[
	\Sha^1(L/k,M) \colonequals 
	\ker\left( \HH^1(\Gal(L/k),M) \to \prod_{v \in \Omega} \HH^1(\Gal(L_w/k_v),M) \right),
\]
where $\Gal(L_w/k_v)$ 
is a decomposition group associated to a chosen place $w$ of $L$ above $v$;
since the conjugation action $\Gal(L/k)$ on itself
induces the identity on $\HH^1(\Gal(L/k),M)$,
it does not matter which $w$ is chosen,
and we could alternatively take the kernel of the map
to the product over \emph{all} $w$
instead of using only one above each $v$.

\begin{lemma}[cf.~\cite{Bruin-Poonen-Stoll-preprint}*{Proposition~8.3}]
\label{L:basic facts}
\hfill
\begin{enumerate}[\upshape (a)]
\item\label{I:H^1cyc=0} 
If a finite group $G$ acts trivially on an abelian group $M$,
then $\HH^1_{\cyc}(G,M)=0$.
\item\label{I:Sha^1 in H^1cyc}
If $L/k$ is a finite Galois extension with Galois group $G$,
and $M$ is a $G$-module, 
then $\Sha^1(L/k,M) \subseteq \HH^1_{\cyc}(G,M)$.
\item\label{I:trivial Sha^1 is 0}
If $L/k$ is a Galois extension with Galois group $G$,
and $G$ acts trivially on an abelian group $M$,
then $\Sha^1(L/k,M)=0$.
\item\label{I:change of field} 
If $L/k$ is a finite Galois extension,
and $M$ is a $\Gal(L/k)$-module, 
and $L'/k$ is a larger Galois extension 
(so $\Gal(L'/L)$ acts trivially on $M$),
then inflation induces an isomorphism
$\Sha^1(L/k,M) \isomto \Sha^1(L'/k,M)$.
\item\label{I:big Sha^1 vs H^1cyc}
If $L'/k$ is a Galois extension,
and $M$ is a finite $\Gal(L'/k)$-module, 
and $G$ is the image of $\Gal(L'/k) \to \Aut M$,
then $\Sha^1(L'/k,M)$ is isomorphic to a subgroup of 
$\HH^1_{\cyc}(G,M)$.
\end{enumerate}
\end{lemma}

\begin{proof}
\hfill
\begin{enumerate}[\upshape (a)]
\item 
A homomorphism $G \to M$ that restricts to $0$ on each cyclic subgroup
of $G$ is $0$.
\item 
By the Chebotarev density theorem,
each cyclic subgroup of $G$ arises as a decomposition subgroup.
\item 
If $L/k$ is finite, 
this follows from \eqref{I:H^1cyc=0} and~\eqref{I:Sha^1 in H^1cyc}.
The general case follows by taking a direct limit.
\item 
{}From the inflation-restriction sequence
\[
	0 \to \HH^1(\Gal(L/k),M) \to \HH^1(\Gal(L'/k),M) \to \HH^1(\Gal(L'/L),M)
\]
mapping to its local analogues, we obtain an exact sequence
\[
	0 \to \Sha^1(L/k,M) \to \Sha^1(L'/k,M) \to \Sha^1(L'/L,M).
\]
The last term is $0$ by~\eqref{I:trivial Sha^1 is 0}.
\item 
The quotient $G$ of $\Gal(L'/k)$ is $\Gal(L/k)$ for 
a finite Galois extension $L/k$.
Apply \eqref{I:change of field}
and then~\eqref{I:Sha^1 in H^1cyc}.
\qedhere
\end{enumerate}
\end{proof}

\begin{proof}[Proof of Proposition~\ref{P:Sha^1} for $\Char k \ne p$]
The case $e=1$ is \cite{Poonen-Rains2012-selmer}*{Proposition~3.3(e)},
so assume $e \ge 2$.
Let $S_e \colonequals \SL_2(\Z/p^e\Z)$.
Let $M \colonequals E[p^e] \isom (\Z/p^e\Z)^2$.
By Lemma~\ref{L:basic facts}\eqref{I:big Sha^1 vs H^1cyc},
$\Sha^1(k,M)$ is isomorphic to a subgroup of $\HH^1_{\cyc}(G,M)$.
The invariant subgroup $M^{S_e}$ is $0$,
so the inflation-restriction sequence for $S_e \le G$
shows that $\HH^1_{\cyc}(G,M) \to \HH^1_{\cyc}(S_e,M)$ is injective.
It remains to show that $\HH^1_{\cyc}(S_e,M)=0$.

The inflation-restriction sequence associated to the central subgroup 
$\{\pm 1\} \le S_e$ is
\begin{equation}
  \label{E:inflation-restriction}
0\To \HH^1(S_e/\{\pm 1\},M[2])
 \stackrel{\inf}\To \HH^1(S_e,M)
 \To \HH^1(\{\pm 1\},M)^{S_e}.
\end{equation}
If $p$ is odd, $M[2]=0$
and $\HH^1(\{\pm 1\},M)=0$ (killed by both $2$ and $p$), 
so $\HH^1(S_e,M)=0$.

So assume that $p=2$.
Then $\HH^1(\{\pm 1\},M) \isom (\Z/2\Z)^2$,
on which $S_e$ acts through $S_1$ in the standard way,
so $\HH^1(\{\pm 1\},M)^{S_e}=0$,
so the map $\inf$ in~\eqref{E:inflation-restriction} is an isomorphism.
The map $\inf$ factors as 
\[
	\HH^1(S_e/\{\pm 1\},M[2]) \To \HH^1(S_e,M[2]) \To \HH^1(S_e,M),
\]
so the second map is surjective.
It is also injective, since $\HH^0(S_e,2M)=0$.
Thus $\HH^1(S_e,M[2]) \isom \HH^1(S_e,M)$.

Define a filtration 
$\{1\} \le \Gamma_{e-1} \le \cdots \le \Gamma_2 \le \Gamma_1 \le S_e$
by $\Gamma_m \colonequals \ker(S_e \to S_m)$.
We prove by induction on $e$ that the inclusion
$\Gamma_1^2 [\Gamma_1,\Gamma_1] \le \Gamma_2$
is an equality.
We check the cases $e=2$ and $e=3$ by hand.
For $e \ge 4$, every element of $\Gamma_{e-1}$ is represented by $1+2^{e-1} A$
for some trace-$0$ integer matrix $A$,
and is the square of $1+2^{e-2} A \in \Gamma_{e-2} \le \Gamma_1$;
now apply the inductive hypothesis to $S_{e-1}=S_e/\Gamma_{e-1}$.

The previous paragraph shows that $\Gamma_2$
is contained in the kernel of every homomorphism $\Gamma_1 \to \Z/2\Z$.
Thus the restriction map $\HH^1(\Gamma_1,M[2]) \to \HH^1(\Gamma_2,M[2])$ is $0$ 
(the actions are trivial).
Consider the maps $\alpha$ and $\beta$ in the inflation-restriction sequence
\[
	0 \To \HH^1(S_2,M[2]) \stackrel{\alpha}\To 
	\HH^1(S_e,M[2]) 
	\stackrel{\beta}\To \HH^1(\Gamma_2,M[2]).
\]
Since $\beta$ factors through the previous restriction map,
$\beta=0$, and $\alpha$ is an isomorphism.
Let $U_e \le S_e$
be the subgroup of unipotent upper triangular matrices.
The horizontal maps in the bottom row of the commutative diagram 
\[
\xymatrix{
\HH^1(S_2,M[2]) \ar[r]^{\sim} \ar[d]^{\res} & \HH^1(S_e,M[2]) \ar[r]^{\sim} \ar[d]^{\res} & \HH^1(S_e,M) \ar[d]^{\res} \\
\HH^1(U_2,M[2]) \ar@{^{(}->}[r]^{\inf} & \HH^1(U_e,M[2]) \ar[r] & \HH^1(U_e,M) \\
}
\]
are injective
(for the second map, observe that $M^{U_e} \stackrel{2}{\to} (2M)^{U_e}$ 
is surjective).
Direct calculation shows that the left vertical map is injective too
(in fact, it is an isomorphism between groups of order $2$).
So the right vertical map is injective.
In particular, $\HH^1_{\cyc}(S_e,M)=0$.
\end{proof}

The following two lemmas will be used in the proof of the 
$\Char k=p$ case of Proposition~\ref{P:Sha^1}.

\begin{lemma}
\label{L:structure of E[p^e]}
Let $k$ be a field of characteristic~$p$.
Let $E$ be an elliptic curve over $k$.
\begin{enumerate}[\upshape (a)]
\item\label{I:structure of ordinary E[p^e]}
If $E$ is ordinary, 
then for any $e \in \Z_{\ge 0}$
there is an exact sequence
\begin{equation}
  \label{E:ordinary E[p^e]}
	0 \to M^\vee \to E[p^e] \to M \to 0
\end{equation}
involving a finite \'etale group scheme $M$ of order $p^e$
and its Cartier dual $M^\vee$.
\item\label{I:structure of supersingular E[p^e]}
If $E$ is supersingular, then $E[p^e]$ is an iterated extension
of copies of $\alpha_p$.
\end{enumerate}
\end{lemma}

\begin{proof}
Let $F \colon E \to E'$ be the $p^e$-Frobenius morphism,
and let $V \colon E' \to E$ be its dual.
Then $F$ is surjective and $VF=p^e$, so there is an exact sequence
\[
	0 \to \ker F \to E[p^e] \to \ker V \to 0.
\]
Moreover, $\ker F$ is the Cartier dual of $\ker V$,
by \cite{MumfordAV1970}*{III.15, Theorem~1}
(the proof there works over any field).
\begin{enumerate}[\upshape (a)]
\item 
If $E$ is ordinary, 
then $\ker V$ is a finite \'etale group scheme of order $\deg V = p^e$.
\item 
Suppose that $E$ is supersingular.
The group scheme $E[p^e]$ is an iterated extension of copies of $E[p]$,
so we may reduce to the case $e=1$.
If $e=1$, then $\ker F$ and $\ker V$ are isomorphic to $\alpha_p$:
over an algebraically closed field, 
this is well known~\cite{Oort1966}*{II.15.5},
and it follows over any field of characteristic~$p$
since the twists of $\alpha_p$ are classified by
$\HH^1(k,\AUT \alpha_p) = \HH^1(k,\G_m) = 0$.\qedhere
\end{enumerate}
\end{proof}

\begin{lemma}
\label{L:Sha^1 of alpha_p and mu_p}
Let $k$ be a global field of characteristic~$p$.
Let $M$ be a finite commutative group scheme over $k$
that is an iterated extension of copies of $\mu_p$ and $\alpha_p$.
If $v \in \Omega$, then $\HH^1(k,M) \to \HH^1(k_v,M)$ is injective.
In particular, $\Sha^1(k,M)=0$.
\end{lemma}

\begin{proof}
When $M=\mu_p$, 
Hilbert's theorem~90 implies that $\HH^1(k,M) \to \HH^1(k_v,M)$ 
is $k^\times/k^{\times p} \to k_v^\times/k_v^{\times p}$.
Similarly, when $M=\alpha_p$, it is the homomorphism of additive groups
$k/k^p \to k_v/k_v^p$.
Both homomorphisms are injective, by \cite{Poonen-Voloch2010}*{Lemma~3.1}.

If $0 \to M' \to M \to M'' \to 0$ is an extension of group schemes
as in the statement, and the result holds for $M'$ and $M''$,
then it holds for $M$ too (this uses injectivity of 
$\HH^1(k_v,M') \to \HH^1(k_v,M)$, which follows since $\HH^0(k_v,M'')=0$).
So the general case follows by induction on $\#M$.
\end{proof}

\begin{proof}[Proof of Proposition~\ref{P:Sha^1} for $\Char k = p$]
\hfill

\emph{Case~1: $E$ is supersingular.}
Combine 
Lemmas \ref{L:structure of E[p^e]}\eqref{I:structure of supersingular E[p^e]} 
and~\ref{L:Sha^1 of alpha_p and mu_p}.

\emph{Case~2: $E$ is ordinary.}
Let $M$ be as in 
Lemma~\ref{L:structure of E[p^e]}\eqref{I:structure of ordinary E[p^e]}.
Let $N = E[p^e](\ksep)$, which injects into $M(\ksep)$
under the map induced by~\eqref{E:ordinary E[p^e]}.
Let $L$ be the splitting field of $M$.
Thus $L$ is a Galois extension of $k$
and the image of $\Gal(L/k) \to \Aut N$ is $G$.
We now break into subcases.

\emph{Case~2a: $L=k$.}
Then \eqref{E:ordinary E[p^e]} has the form
\[
	0 \to \mu_{p^e} \to E[p^e] \to \Z/p^e\Z \to 0.
\]
By Lemma~\ref{L:basic facts}\eqref{I:trivial Sha^1 is 0}, 
$\Sha^1(k,\Z/p^e\Z)=0$,
so any $\xi \in \Sha^1(k,E[p^e])$ 
must come from an element $\eta \in \HH^1(k,\mu_{p^e})$.
Pick any $v \in \Omega$.
The middle vertical map in the commutative diagram
\[
\xymatrix{
\Z/p^e\Z \ar[r] \ar@{=}[d] & \HH^1(k,\mu_{p^e}) \ar[r] \ar[d] & \HH^1(k,E[p^e]) \ar[d] \\
\Z/p^e\Z \ar[r] & \HH^1(k_v,\mu_{p^e}) \ar[r] & \HH^1(k_v,E[p^e]) \\
}
\]
is injective by Lemma~\ref{L:Sha^1 of alpha_p and mu_p},
and a diagram chase shows that $\eta$ comes from an element of $\Z/p^e\Z$.
Thus $\xi=0$.

\emph{Case~2b: $L$ is general.}
By definition of $L$, we have $E[p^e](L)=N$.
For any place $w$ of $L$, 
every element of $L_w$ that is algebraic over $L$
is actually separable over $L$ \cite{Poonen-Voloch2010}*{Lemma~3.1},
so $E[p^e](L_w)=N$ too.
By Case~2a, $\Sha^1(L,E[p^e])=0$.
Because of the (fppf) inflation-restriction sequence
\[
	0 \to \HH^1(\Gal(L/k),N)
	\to \HH^1(k,E[p^e])
	\to \HH^1(L,E[p^e]),
\]
which maps to its analogue for each extension $L_w/k_v$ of local fields,
we have $\Sha^1(k,E[p^e]) \isom \Sha^1(\Gal(L/k),N)$.
By Lemma~\ref{L:basic facts}\eqref{I:big Sha^1 vs H^1cyc},
the latter is isomorphic to a subgroup of $\HH^1_{\cyc}(G,N)$,
which is trivial since $G$ is cyclic by assumption.
\end{proof}

\begin{remark}
In Proposition~\ref{P:Sha^1}, when $p=2$ and $e=3$,
the hypothesis that the image of $G_k \to \Aut E[p^e](\ksep)$
is cyclic can fail (but only for $0\%$ of $E \in \EE$,
as explained in Remark~\ref{R:hypothesis holds}).
The last line of the proof above 
cannot be immediately extended to the case
in which the image is non-cyclic,
because one can check that $\HH^1_{\cyc}((\Z/8\Z)^\times,\Z/8\Z) \ne 0$
for the standard nontrivial action.
The \emph{conclusion} of Proposition~\ref{P:Sha^1} might still hold, 
however.
\end{remark}

\subsection{Intersection of maximal isotropic subgroups}

For nonarchimedean $v$, let $\calO_v$ be the valuation ring in $k_v$.
Let $\Adeles = \sideset{}{'}\prod_{v \in \Omega} (k_v,\calO_v)$
be the adele ring of $k$.
Suppose that $E$, $k$, and $p^e$ satisfy the hypothesis of 
Proposition~\ref{P:Sha^1}, so that $\Sha^1(k,E[p^e])=0$.
Then \cite{Poonen-Rains2012-selmer}*{Theorem~4.14} applied with 
$\lambda \colon A \to \widetilde{A}$ being $[p^e] \colon E \to E$
shows that 
$\Sel_{p^e} E$ is isomorphic to 
the intersection of two maximal isotropic subgroups of 
\[
	\HH^1(\Adeles,E[p^e]) \colonequals 
	\sideset{}{'}\prod_{v \in \Omega} 
	(\HH^1(k_v,E[p^e]),\HH^1(\calO_v,E[p^e]))
	\isom
	\sideset{}{'}\prod_{v \in \Omega} 
	\left(\HH^1(k_v,E[p^e]),\frac{E(k_v)}{p^e E(k_v)}\right),
\]
namely the images of $E(\Adeles)/p^e E(\Adeles) = \prod_v E(k_v)/p^e E(k_v)$ 
and $\HH^1(k,E[p^e])$.

\subsection{Direct summands}

It is natural to ask whether these images are direct summands,
given that we \emph{modeled} $\Sel_{p^e} E$ 
by an intersection of direct summands.
Corollary~\ref{C:adelic direct summand} below
shows that at least the first of these images is a direct summand.

\begin{proposition}
\label{P:Skinner}
Let $E$ be an abelian variety over an arbitrary field $k$.
Let $n \in \Z_{>0}$.
Then the image of the coboundary map 
$E(k)/nE(k) \stackrel{\delta}\to \HH^1(k,E[n])$ is a direct summand
of $\HH^1(k,E[n])$.
\end{proposition}

\begin{proof}
(We thank Bart de Smit and Christopher Skinner for ideas used in this proof.)
For each $m|n$, the commutative diagram 
\[
\xymatrix{
\HH^1(k,E[m]) \ar@{->>}[r] \ar[d] & \HH^1(k,E)[m] \ar@{^{(}->}[d] \\
\HH^1(k,E[n]) \ar@{->>}[r]^{\alpha} & \HH^1(k,E)[n] \\
}
\]
shows that any order $m$ element of $\HH^1(k,E)[n]$
lifts to an order $m$ element of $\HH^1(k,E[n])$
under the surjection $\alpha$ in the diagram.
Any $\Z/n\Z$-module is a direct sum of cyclic groups \cite{Pruefer1923}*{\S17}, \cite{Baer1935}*{pp.~274--275};
applying this to $\HH^1(k,E)[n]$
and using the previous sentence shows that $\alpha$ is split.
Finally, $\ker(\alpha)=\im(\delta)$.
\end{proof}

\begin{corollary}
\label{C:adelic direct summand}
Let $E$ be an abelian variety over a global field $k$.
Let $n \in \Z_{>0}$.
Then the image of 
$E(\Adeles)/nE(\Adeles) \stackrel{\delta}\to \HH^1(\Adeles,E[n])$ 
is a direct summand of $\HH^1(\Adeles,E[n])$.
\end{corollary}

\begin{proof}
Proposition~\ref{P:Skinner} yields a complement $C_v$ 
of $E(k_v)/nE(k_v)$ in $\HH^1(k_v,E[n])$.
Then $\Directsum_{v \in \Omega} C_v$ is a complement 
of $E(\Adeles)/nE(\Adeles)$ in $\HH^1(\Adeles,E[n])$.
\end{proof}

Is the other subgroup,
the image of $\HH^1(k,E[n]) \to \HH^1(\Adeles,E[n])$, 
a direct summand too?
Lemma~\ref{L:Chebotarev direct summand} below gives a positive answer
for some elliptic curves.
Although it applies only to $0\%$ of $E \in \EE$,
it may be that the answer is positive for all $E$.
We conjecture at least the following.

\begin{conjecture}
\label{C:global H^1 is direct summand}
Fix a global field $k$ and $n \ge 1$.
The image of $\HH^1(k,E[n]) \to \HH^1(\Adeles,E[n])$ is a direct summand
for $100\%$ of $E \in \EE$.
\end{conjecture}

\begin{lemma}
\label{L:Chebotarev direct summand}
If $\Char k \nmid n$ and the action of $G_k$ on $E[n]$ is trivial,
then the image of $\HH^1(k,E[n]) \to \HH^1(\Adeles,E[n])$ is a direct summand.
\end{lemma}

\begin{proof}
We have $E[n] \isom \mu_n \times \mu_n$,
so we must show that the image of 
$k^\times/k^{\times n} \to \Adeles^\times/\Adeles^{\times n}$
is a direct summand.
By Lemma~\ref{L:snake} below, it is enough to show that
$k^\times/k^{\times m} \to \Adeles^\times/\Adeles^{\times m}$
is injective for each $m|n$.
This ``local-global principle for $m^{\tH}$ powers''
is a well known consequence of the Chebotarev density theorem.
\end{proof}

\begin{lemma}
\label{L:snake}
Let $n \in \Z_{>0}$.
Let $\delta \colon A \to B$ be a homomorphism of $\Z/n\Z$-modules
such that the induced morphism $A/mA \to B/mB$ is injective
for every $m|n$.
Then $\delta(A)$ is a direct summand of $B$.
\end{lemma}

\begin{proof}
(We thank Bart de Smit for this proof.)
Taking $m=n$ shows that $\delta$ is injective, 
so it fits into a short exact sequence
\[
	0 \to A \stackrel{\delta}\to B \to C \to 0
\]
of $\Z/n\Z$-modules.
Write $C$ as a direct sum of cyclic groups $C_i$.
For each $m$, the hypothesis together with the snake lemma
shows that $B[m] \to C[m]$ is surjective.
Thus we can construct a splitting of the surjection $B \to C$,
by lifting a generator of each $C_i$
to an element of $B$ of the same order.
\end{proof}

\begin{remark}
\label{R:why direct summands}
Suppose that for each $n$ we sample $Z$ and $W$
from a distribution on maximal isotropic subgroups of $(\Z/p^e\Z)^{2n}$
that is not necessarily supported on direct summands,
but still invariant under $\Orthogonal_{2n}(\Z/p^e\Z)$.
If $Z \intersect W$ models $\Sel_{p^e} E$,
then $\dim (Z \intersect W)[p]$ should model $\Sel_p E$,
and in particular should have the distribution 
predicted by and justified by \cite{Poonen-Rains2012-selmer}.
We will show that this happens only 
if the probability of $Z$ and $W$ being direct summands of $(\Z/p^e\Z)^{2n}$
tends to $1$ as $n \to \infty$.

If $(Z \intersect W)[p]=0$, then $Z \intersect W=0$,
so the homomorphism $Z \directsum W \to (\Z/p^e\Z)^{2n}$
between groups of equal size is an isomorphism;
i.e., $Z$ and $W$ are direct summands.
Therefore
\begin{align*}
	\Prob((Z \intersect W)[p]=0)
	\le & \Prob(\textup{$Z,W$ are direct summands}) \\
	& \cdot \Prob(Z \intersect W=0 \mid \textup{$Z,W$ are direct summands}).
\end{align*}
But $\lim_{n \to \infty} \Prob(Z \intersect W=0 \mid \textup{$Z,W$ are direct summands})$
equals the desired limiting value of $\Prob((Z \intersect W)[p]=0)$,
which is nonzero,
so $\Prob(\textup{$Z,W$ are direct summands})$ must tend to $1$.
\end{remark}

Remark~\ref{R:why direct summands} may be viewed as indirect evidence
for Conjecture~\ref{C:global H^1 is direct summand}.

\subsection{Freeness of the ambient group}

On the arithmetic side,
the $\Z/p^e\Z$-module $\HH^1(\Adeles,E[p^e])$ carrying a quadratic form 
is not always free.
But we have modeled it by the free module $(\Z/p^e\Z)^{2n}$ 
with $(Q \bmod p^e)$.

\begin{question}
Can we develop a more sophisticated model 
in which we start with a compatible system 
consisting of a quadratic form on a non-free $\Z/p^e\Z$-module
for each $e$?
\end{question}

Given the compatibility of our model with known theorems and conjectures,
we expect that incorporating non-freeness into the model would
not change the distribution constructed in Section~\ref{S:intersection}.

\section*{Acknowledgements}

We thank 
K\k{e}stutis \v{C}esnavi\v{c}ius, 
Bart de Smit, 
Christophe Delaunay, 
and
Christopher Skinner 
for comments.
This research was begun during the ``Arithmetic Statistics'' semester
at the Mathematical Sciences Research Institute,
and continued during the ``Cohen-Lenstra heuristics
for class groups'' workshop at the American Institute of Mathematics,
the 2012 Canadian Number Theory Association meeting 
at the University of Lethbridge,
the Centre Interfacultaire Bernoulli semester on 
``Rational points and algebraic cycles'',
the 2013 ``Explicit methods in number theory'' workshop 
at the Mathematisches Forschungsinstitut Oberwolfach,
and the ``Rational points 2013'' workshop at Schloss Thurnau.

\end{document}